\documentclass[11pt,psamsfonts]{article}

\usepackage{a4,amssymb,times,fullpage}

\usepackage{float}
\usepackage{ucs}
\usepackage[utf8]{inputenc}
\usepackage[T1]{fontenc}
\usepackage{color}
\usepackage{graphicx}
\usepackage{amsmath}
\usepackage{amsfonts}
\usepackage{enumerate}
\usepackage{latexsym}
\usepackage{amsthm}

\makeatletter\@addtoreset{equation}{section}\makeatother
\makeatletter\@addtoreset{figure}{section}\makeatother
\makeatletter\@addtoreset{table}{section}\makeatother

\newtheorem{theorem}{Theorem}[section]
\newtheorem{prop}[theorem]{Proposition}
\newtheorem{lemma}[theorem]{Lemma}

\newtheorem{conj}[theorem]{Conjecture}

\newtheorem{cor}[theorem]{Corollary}

\newcommand{\R}{{\mathbb R}}
\newcommand{\C}{{\mathbb C}}
\newcommand{\Z}{{\mathbb Z}}

\newcommand{\N}{{\mathbb N}}

\newcommand{\norm}[1]{\left\|#1\right\|}

\newcommand{\abs}[1]{\left|#1\right|}

\newcommand{\op}[1]{\!\!\mathop{\rm ~#1}\nolimits}
\newcommand{\DD}{\!\mathop{\rm d\!}\nolimits}

\newcommand{\scriptop}[1]{\!\!\mathop{\mbox{\rm \scriptsize ~#1}}\nolimits}

\newcommand{\deriv}[2]{\frac{\partial #1}{\partial #2}}
\newcommand{\ii}{\textup{i}}

\newenvironment{remark}{\refstepcounter{theorem}\par\medskip\noindent{\bf
Remark~\thetheorem~~}}{\unskip\nobreak\hfill\hbox{ $\oslash$}\par\bigskip}

\newenvironment{definition}{\refstepcounter{theorem}\par\medskip\noindent{\bf
Definition~\thetheorem~~}}{\unskip\nobreak\hfill\hbox{ $\oslash$}\par\bigskip}

\title{ Hamiltonian dynamics and spectral theory for
  spin\--oscillators} 

\author{\'Alvaro Pelayo\thanks{Partially supported by an NSF
    Postdoctoral Fellowship.} \,\, and San V\~u Ng\d
  oc\thanks{Partially supported by an ANR 'Programme Blanc'.}}
\date{}
\begin{document}
\maketitle

\begin{abstract}

  We study the Hamiltonian dynamics and spectral theory of
  spin\--oscillators.  Because of their rich structure,
  spin\--oscillators display fairly general properties of integrable
  systems with two degrees of freedom.  Spin\--oscillators have
  infinitely many transversally elliptic singularities, exactly one
  elliptic-elliptic singularity and one focus\--focus singularity.
  The most interesting dynamical features of integrable systems, and
  in particular of spin\--oscillators, are encoded in their
  singularities.  In the first part of the paper we study the
  symplectic dynamics around the focus\--focus singularity.  In the
  second part of the paper we quantize the coupled spin\--oscillators
  systems and study their spectral theory.  The paper combines
  techniques from semiclassical analysis with differential geometric
  methods.
  \end{abstract}


\section{Introduction}

Coupled spin\--oscillators are $4$\--dimensional integrable
Hamiltonian systems with two degrees of freedom constructed by
``coupling'' the classical spin on the $2$\--sphere $S^2$ (see Figure
\ref{fig:spin}) with the classical harmonic oscillator on the
Euclidean plane $\mathbb{R}^2$.  Coupled spin\--oscillators are one of
the most fundamental examples of integrable systems; their dynamical
behavior is rich and represents some fairly general properties of low
dimensional integrable systems. The goal of this paper is to study
coupled spin\--oscillators from the point of view of classical and
quantum mechanics, using methods from classical and semiclassical
analysis.

A $4$\--dimensional integrable system with two degrees of freedom
consists of a connected symplectic $4$\--manifold equipped with two
almost everywhere linearly independent smooth functions which Poisson
commute, i.e. two smooth functions on the manifold such that one of
them is invariant along the flow of the Hamiltonian vector field
generated by of the other.  The most interesting geometric and
dynamical features of integrable systems are encoded in their
singularities, i.e the points where Hamiltonian vector fields
generated by the functions are linearly dependent.  Around the regular
points, the dynamics is simple, and described by the
Arnold\--Liouville\--Mineur action\--angle theorem. As we will see,
the dynamics near the singularities is in general much more
complicated and depends heavily on the type of singularity.

Let us explain the construction of coupled spin\--oscillators more precisely.  
Let $S^2$ be the unit sphere in $\R^3$ with coordinates $(x,\,y,\,z)$,
and let $\R^2$ be equipped with coordinates $(u,\, v)$. Let $\lambda, \rho>0$
be positive constants. Let $M$ be
the product manifold $S^2\times\R^2$ equipped with the product
symplectic structure $\lambda \omega_{S^2} \oplus \rho \omega_0$.  Let $J,\,H
\colon M \to \R$ be the smooth maps defined by $ J := \rho (u^2+v^2)/2 + \lambda z
$ and $H := \frac{1}{2} \, (ux+vy)$. 
A \emph{coupled spin\--oscillator} is
a $4$\--dimensional integrable system of the form $(M,\, \lambda \omega_{S^2} \oplus
\rho \omega_0,\, (J,\,H))$, where $\omega_{S^2}$ is the standard
symplectic form on the sphere and $\omega_0$ is the standard
symplectic form on $\R^2$.

The singularities of coupled spin\--oscillators are non\--degenerate
and of elliptic\--elliptic, transversally\--elliptic (both of these
types are usually referred to as ``elliptic singularities'') or
focus\--focus type.  They have infinitely many
transversally\--elliptic singularities (along a piecewise smooth
curve, as we shall see), one elliptic\--elliptic singularity at
$(0,0,-1,0,0)$ and one singularity of focus\--focus type at
$(0,0,1,0,0)$.  The $J$ component of this system is the Hamiltonian
(or momentum map) of the $S^1$\--action that simultaneously rotates
about the vertical axes of the $2$\--sphere, and about the origin of
$\mathbb{R}^2$.  The $H$ component is given as follows. Using the
natural embedding of $S^2$ in $\mathbb{R}^3$, let $\pi_z$ be the
orthogonal projection from $S^2$ onto $\mathbb{R}^2$ viewed as the
$z=0$ hyperplane. Let $(x,\,y,\,z) \in S^2$ and $(u,\, v) \in
\mathbb{R}^2$. Under the flow of $J$ the points $(x,\,y,\,z)$ and
$(u,\,v)$ are moving along the flows of $z$ and $(u^2+v^2)/2$,
respectively, with the same angular velocity.  Hence the inner product
$\langle \pi_z(x,\,y,\,z),\, (u,\,v) \rangle=ux+vy=2H$ is constant and
commutes with $J$.

Because $H$ does not come from an $S^1$\--action, coupled
spin\--oscillators are not toric integrable systems -- they are what now is called \emph{semitoric
integrable systems}, or simply \emph{semitoric systems}.  Semitoric systems form
a rich class of integrable systems, commonly found in simple
physical models.  For simplicity, throughout this paper we assume 
the rescaling $\lambda=\rho=1$. The statements and proofs extend immediately to the case
of $\lambda,\, \rho>0$, but we feel that the notation is already
sufficiently heavy so we shall avoid carrying these parameters.

\subsection*{Semitoric integrable systems}

Our interest in semitoric integrable systems was motivated by the
remarkable convexity results for Hamiltonian torus actions by Atiyah
\cite{atiyah-convex}, Guillemin\--Sternberg
\cite{guillemin-sternberg}, and Delzant \cite{delzant}. Despite
important contributions by Arnold, Duistermaat \cite{duistermaat},
Eliasson \cite{eliasson-these}, V\~u Ng\d oc \cite{san-semi-global,
  san-polytope}, Zung \cite{zung-I} and many others, the singularity
theory of integrable systems from the point of view of symplectic
geometry is far from being completely understood.  As a matter of
fact, very few integrable systems are understood. The singularities 
of these systems encode a vast amount of information
about the symplectic dynamics and geometry of the system, much of which is
not computable with the current methods. 
 
This singularity theory is interesting not only from the point of view
of semiclassical analysis and symplectic geometry, but it also shares
many common features with the study of singularities in the context of
symplectic topology \cite{symington-four, leung-symington}, algebraic
geometry and mirror symmetry (see \cite{gross-siebert1} and the
references therein).

The coupled spin\--oscillator is perhaps the simplest non\--compact
example of an integrable system of semitoric type. Precisely, a {\em
  semitoric integrable system} on $M$ is an integrable system $J, \, H
\in \op{C}^{\infty}(M,\, \R)$ for which the component $J$ is a proper
momentum map for a Hamiltonian circle action on $M$ and the map
$F:=(J,\,H):M\to\R^2$ has only non\--degenerate singularities in the
sense of Williamson \cite{williamson}, without real-hyperbolic
blocks. This means that in addition to the well\--known elliptic
singularities of toric systems, semitoric systems may have
\emph{focus\--focus singularities}.
  
Semitoric integrable systems on $4$\--manifolds have been symplectically 
classified by the authors in \cite{san-alvaro-I, san-alvaro-II} in terms a
collection of five invariants.  While conceptually they are more easily describable,
some of these invariants are involved to compute explicitly for a particular
integrable system.  The most difficult invariant to
compute is the so called Taylor series invariant, which classifies a
neighborhood of the \emph{focus\--focus singular fiber} of
$F$.  This invariant, which was introduced in \cite{san-semi-global},
encodes a large amount of information about the
local and semiglobal behavior of the system. 
Focus\--focus singular fibers are singular fibers that contain
some fixed point $m$ (i.e. $\op{rank}(\op{d}\! F)=0$) which is of
\emph{focus\--focus} type, meaning that there are symplectic coordinates
locally near $m$ in which $m=(0,0,0,0)$, $\omega=\op{d}\!  \xi \wedge
\op{d}\!x +\op{d}\! \eta\wedge \op{d}\!y$ and $F=F(m)+(x\xi+y\eta, \,
x\eta-y\xi) +\mathcal{O}((x,\, \xi,\, y, \,\eta)^3)$.

\subsection*{Dynamics and singularities of coupled spin\--oscillators}

The coupled spin\--oscillator system has non\--degenerate
singularities of elliptic\--elliptic, transversally\--elliptic and
focus\--focus type. It has exactly one singularity of focus\--focus
type.  Near the focus\--focus singularity, the behavior of the
Hamiltonian vector fields generated by the system is not
$2\pi$\--periodic, as it occurs with toric systems.

\begin{figure}[h]
  \begin{center}
    \includegraphics[width=7cm]{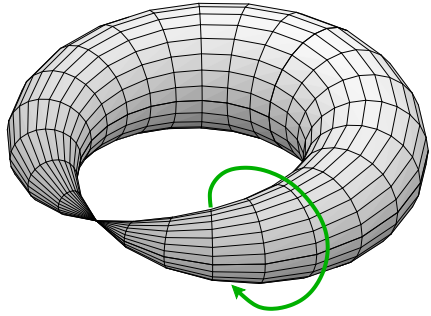}
    \caption{Singularity of focus\--focus type and vanishing cycle.
      Topologically a fiber containing a single focus\--focus
      singularity is a pinched torus.}
    \label{fig:pinched}
  \end{center}
\end{figure}

Loosely speaking, one of the components of the system is indeed
$2\pi$\--periodic, but the other one generates an arbitrary flow which
turns indefinitely around the focus\--focus singularity and which, as
$F$ tends to the critical value $F(m)$, deviates from periodic
behavior in a logarithmic fashion, up to a certain error term; this
deviation from being logarithmic is a symplectic invariant and can be
made explicit -- it is in fact given by an infinite Taylor series
$(S)^{\infty}$ on two variables $X,Y$ with vanishing constant term.
This was proven by the second author in \cite{san-semi-global}.  The
goal of the first part of the present paper is compute the linear
approximation of this deviation.

\begin{theorem} \label{main} The coupled spin--oscillator is a
  semitoric integrable system, with one single focus\--focus
  singularity at $m=(0,\,0,\,1,\,0,\,0) \in S^2 \times \R^2$.  The
  semiglobal dynamics around $m$ may be described as follows: the
  linear deviation from exhibiting logarithmic behavior in a saturated
  neighborhood of $m$ is given by the linear map $L \colon \R^2 \to
  \R$ with expression $L(X,\, Y)=\frac{\pi}{2} \, X+ 5\ln 2\, Y.$ In
  other words, we have an equality $(S(X,\,Y))^{\infty}=L(X,\,Y)
  +\mathcal{O}(X,Y)^2,$ where $(S(X,\,Y))^{\infty}$ denotes the Taylor
  series invariant at the focus\--focus singularity.
\end{theorem}

As far as we know, this theorem gives the first rigorous estimate in the literature
of the logarithmic deviation, and hence the first explicit quantization of the
symplectic dynamics around the singularity; we prove it in Section 2.  The proof is
computational but rather subtle, and it combines a number of theorems from
integrable systems and semiclassical analysis.  The method of proof of Theorem 
\ref{main} (given
in several steps) provides a fairly general algorithm to implement in
the case of other semitoric integrable systems. Moreover, it seems plausible
to expect that the techniques we introduce generalize to compute higher
order approximations, but not immediately -- indeed, the linear approximation
relies on various semiclassical formulas that are not readily available for
higher order approximations. In this paper we will
also find the other invariants that characterize the coupled
spin\--oscillator (Section 3): the polygon and height invariants;
these are easier to find.

\subsection*{Spectral theory for quantum coupled spin\--oscillators}

Sections 4, 5 of this paper are devoted to the spectral theory of
quantum coupled spin\--oscillators.  The following theorem describes
the quantum spin\--oscillator.  For any $\hbar>0$ such that
$2=\hbar(n+1)$, for some non-negative integer $n\in\N$, let
$\mathcal{H}$ denote the standard $n+1$-dimensional Hilbert space
quantizing the sphere $S^2$ (see
Section~\ref{sec:harmonic_quantization}).

 \begin{figure}[h]
   \centering
   \includegraphics[width=0.3\linewidth]{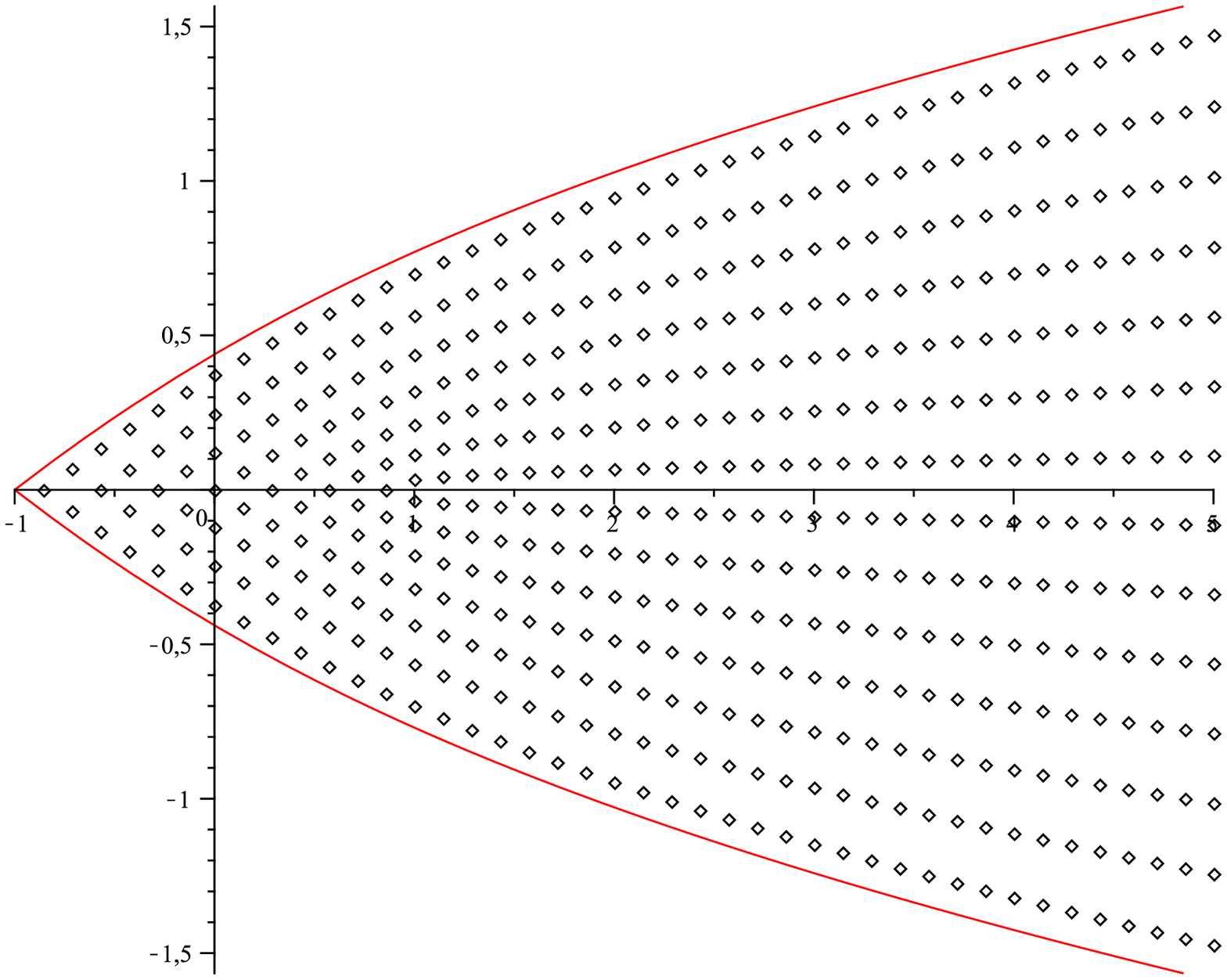}
   \caption{Semiclassical joint spectrum of $\hat{J}, \hat{H}$. We will
   explain this figure in more detail in Section 4.}
   \label{fig:spectrumapprox2versionsmall}
 \end{figure}

 \begin{theorem} \label{thm:spectral} Let $S^2 \times \mathbb{R}^2$ be
   the coupled spin\--oscillator, and (as above) let $J,\, H \colon M
   \to \mathbb{R}$ be the Poisson commuting smooth functions that
   define it. The unbounded operators $\hat{J}:=\op{Id} \otimes \Big(
   -\frac{\hbar^2}{2} \frac{\op{d}^2}{\op{d}u^2} +\frac{u^2}{2} \Big)
   + (\hat{z} \otimes \op{Id})$ and
   $\hat{H}=\frac{1}{2}(\hat{x}\otimes u + \hat{y} \otimes
   (\frac{\hbar}{\ii}\frac{\partial}{\partial u})$ on the Hilbert
   space $ \mathcal{H} \otimes \op{L}^2(\R)\subset \op{L}^2(\R^2)
   \otimes \op{L}^2(\R)$ are self\--adjoint and commute. The spectrum
   of $\hat{J}$ is discrete and consists of eigenvalues in
   $\hbar(\frac{1-n}{2}+\N)$.

   For a fixed eigenvalue $\lambda$ of $\hat{J}$, let
   $\mathcal{E}_{\lambda}:=\op{ker}(\hat{J}-\lambda \op{Id})$ be the
   eigenspace of the operator $\hat{J}$ over $\lambda$.  There exists
   a basis $\mathcal{B}_{\lambda}$ of $\mathcal{E}_{\lambda}$ in which
   $\hat{H}$ restricted to $\mathcal{E}_{\lambda}$ is given by
  \begin{eqnarray*}
    \op{M}_{\mathcal{B}_{\lambda}}(\hat{H})=
    \Big(\frac{\hbar}{2}\Big)^{\frac{3}{2}}\,
    \left( \begin{array}{ccccccc}
        0   & \beta_1 & \dots & &&&0 \\
        \beta_1     & 0      &  \beta_2 &      &               &&  0\\
        0 &\beta_2&   0 &\beta_3               && & 0\\
        & &     & &\\
        \vdots & \vdots & \ddots  &\vdots &&\vdots &\vdots\\
        & &                     &  &&&\beta_{\mu} \\
        0& 0& \dots & & & \beta_{\mu}& 0
      \end{array} \right),
  \end{eqnarray*}
  where  $0 \le k \le n$,
  $\ell_0:=\frac{\lambda}{\hbar}+\frac{n-1}{2}$,
  $\mu:=\op{min}(\ell_0,n)$, $\beta_k:=\sqrt{(\ell_0+1-k)k(n-k+1)}$.   
 
  The dimension of $\mathcal{E}_{\lambda}$ is $\mu+1$.
\end{theorem}

Finding out how information from quantum completely integrable systems
leads to information about classical systems is a fascinating
``inverse'' problem with very few precise results at this time.
Section 5 explains how information of the coupled spin\--oscillator,
including its \emph{linear} singularity theory (computed in Section
2), may be recovered from the quantum semiclassical spectrum.

The way in which we recover this linear singularity theory relies on a
conjecture for Toeplitz operators, which has been proven 
for pseudodifferential operators.  We explain in detail how to do this
and formulate the following conjecture about semitoric integrable
systems: that a semitoric system is determined up to symplectic
equivalence by its semiclassical joint spectrum, i.e.  the set of points 
in $\mathbb{R}^2$ where on the $x$\--axis we
have the eigenvalues of $\hat{J}$, and on the vertical axis the
eigenvalues of $\hat{H}$ restricted to the $\lambda$\--eigenspace of
$\hat{J}$.  From any such
spectrum one can construct explicitly the associated semitoric system.
We give strong evidence of this conjecture for the coupled
spin oscillators.

\paragraph{Acknowledgements.} 
The work on this article started during a short but intense visit of
the second author to Berkeley. He is grateful to the Berkeley maths
department, and in particular to Alan Weinstein and Maciej Zworski for
their invitation.  Part of this paper was written while the first
author was a Professeur Invit\'e in the \'Equations aux D\'eriv\'ees
Partielles Section at the Institut de Recherches Mathématiques de
Rennes (Universit\'e Rennes 1) during January 2010, and he thanks them
for the warm hospitality. He also thanks MSRI for hospitality during
the Fall of 2009 and Winter 2010 when we has a member, and the
University of Paris\--Orsay for their hospitality during the author's
visit on February 2010, during which a portion of this paper was
written.

\section{Singularity theory for coupled spin\--oscillators}
\label{taylor:sec}

This section considers semiglobal properties. It is independent of
Section 3 which concerns global properties. The main goal of this section
is to prove Theorem \ref{main}.

Let $(M,\, \omega,\, F:=(J,\,H))$ be a semitoric integrable system.
Recall that a \emph{singular point}, or a \emph{singularity}, is a
point $p \in M$ such that $\,\op{rank}(\op{d}\!F)(p)<2$, where
$F:=(J,\,H) \colon M \to \R^2$.  A \emph{singular fiber} of the system
is a fiber of $F \colon M \to \mathbb{R}^2$ that contains some
singular point. 

Let $m$ be a focus\--focus singular point $m$.  Let $B:=F(M)$.  Let
$\tilde{c}=F(m)$.  The set of regular values of $F$ is
$\op{Int}(B)\setminus\{\tilde{c}\}$, the boundary of $B$ consists of
all images of elliptic singularities, and the fibers of $F$ are
connected (see \cite{san-polytope}).
  
We assume that the critical fiber $ \mathcal{F}_m:=F^{-1}(\tilde{c}) $
contains only one critical point $m$, which according to Zung
\cite{zung} is a generic condition, and let $\mathcal{F}$ denote the
associated singular foliation.
   
By Eliasson's theorem \cite{eliasson-these} there exist symplectic
coordinates $(x_1,\, x_2,\, \xi_1,\,\xi_2)$ in a neighborhood $U$
around $m$ in which $(q_1,\,q_2)$, given by
\begin{equation}
   q_1=x_1\xi_2-x_2\xi_1, \,\, q_2=x_1\xi_1+x_2\xi_2,
  \label{equ:cartan}
\end{equation}
is a momentum map for the foliation $\mathcal{F}$ (in the sense that
for some local diffeomorphism $q=g \circ F$, so the maps $q$ and $F$
have the same fibers); here the critical point $m$ corresponds to
coordinates $(0,\,0,\,0,\,0)$. Because of the uniqueness of the $S^1$\--action
one may chose Eliasson's coordinates \cite{san-focus} such that $q_1=J$.

\subsection{Construction of the singularity invariant at a focus\--focus singularity}
Fix $A'\in \mathcal{F}_m\cap (U\setminus\{m\})$ and let $\Sigma$
denote a small 2\--dimensional surface transversal to $\mathcal{F}$ at
the point $A'$, and let $\Omega$ be the open neighborhood of
$\mathcal{F}_m$ which consists of the leaves which intersect the
surface $\Sigma$.
  
\begin{figure}[h]
  \begin{center}
    \includegraphics{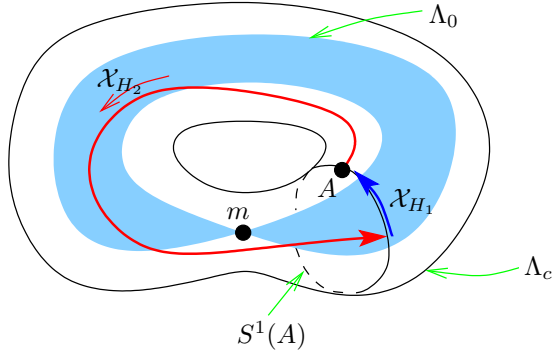}
    \caption{Singular foliation near the leaf $\mathcal{F}_m$, where
      $S^1(A)$ denotes the $S^1$\--orbit generated by $H_1=J$.}
  \end{center}
  \label{AF6}
\end{figure}

Since the Liouville foliation in a small neighborhood of $\Sigma$ is
regular for both $F$ and $q=(q_1,\,q_2)$, there is a local
diffeomorphism $\varphi$ of $\R^2$ such that $q=\varphi \circ {F}$,
and we can define a global momentum map $\Phi=\varphi \circ{F}$ for
the foliation, which agrees with $q$ on $U$.  Write
$\Phi:=(H_1,\,H_2)$ and $\Lambda_c:=\Phi^{-1}(c)$.  For simplicity we
write $\Phi=q$.  Note that $ \Lambda_0=\mathcal{F}_m.$ It follows
from~(\ref{equ:cartan}) that near $m$ the $H_1$\--orbits must be
periodic of primitive period $2\pi$.

Suppose that $A \in\Lambda_c$ for some regular value $c$.  Let
$\tau_2(c)>0$ be the time it takes the Hamiltonian flow associated
with $H_2$ leaving from $A$ to meet the Hamiltonian flow associated
with $H_1$ which passes through $A$, and let $\tau_1(c)\in\R/2\pi\Z$
the time that it takes to go from this intersection point back to $A$,
hence closing the trajectory. We denote by $\gamma_c$ the
corresponding loop in $\Lambda_c$.
  
Write $c=(c_1,\,c_2)=c_1+\ii c_2$, and let $\op{ln} z$ for a fixed
determination of the logarithmic function on the complex plane. Let
\begin{equation}
  \left\{
    \begin{array}{ccl}
      \sigma_1(c) & = & \tau_1(c)-\Im(\op{ln} c) \\
      \sigma_2(c) & = & \tau_2(c)+\Re(\op{ln} c),
    \end{array}
  \right.
  \nonumber
\end{equation}
where $\Re$ and $\Im$ respectively stand for the real an imaginary
parts of a complex number.  V\~u Ng\d oc proved in
\cite[Prop.\,3.1]{san-semi-global} that $\sigma_1$ and $\sigma_2$
extend to smooth and single\--valued functions in a neighbourhood of
$0$ and that the differential 1\--form
  $$
  \sigma:=\sigma_1\, \DD{}c_1+\sigma_2\, \DD{}c_2
  $$
  is closed.  Notice that if follows from the smoothness of $\sigma_2$
  that one may choose the lift of $\tau_2$ to $\R$ such that
  $\sigma_2(0)\in[0,\,2\pi)$. This is the convention used throughout.
  Following \cite[Def.~3.1]{san-semi-global} , let $S$ be the unique
  smooth function defined around $0\in\R^2$ such that
  $$\DD{}S=\sigma,\,\, \,\, S(0)=0.$$
  The Taylor expansion of $S$ at $(0,\,0)$ is denoted by $(S)^\infty$.

  The Taylor expansion $(S)^{\infty}$ is a formal power series in two
  variables with vanishing constant term, and we say that $(S)^\infty$
  is the \emph{Taylor series invariant of $(M,\, \omega,\, (J,\,H))$
    at the focus\--focus point $c$}.

  \subsection{The coupled spin\--oscillators}

  Let $S^2$ be the unit sphere in $\R^3$ with coordinates
  $(x,\,y,\,z)$, and let $\R^2$ be equipped with coordinates $(u,\,
  v)$.  Recall from the introduction that the coupled\--spin
  oscillator model is the product $S^2\times\R^2$ equipped with the
  product symplectic structure $\omega_{S^2} \oplus \omega_0$ given by
  $\op{d}\!\theta \wedge \op{d}\!z \oplus \op{d}\!u \wedge \op{d}\!v$,
  and with the smooth Poisson commuting maps $J,\,H \colon M \to \R$
  given by $ J := (u^2+v^2)/2 + z $ and $H := \frac{1}{2} \, (ux+vy)$.
  Sometimes we denote the coupled spin\--oscillator by the triple
  $(S^2 \times \R^2,\, \omega_{S^2}\oplus\omega_0,\, (J,\,H))$.  A
  simple verification leads to the following observation.

  \begin{prop} \label{prop:nd} The coupled spin--oscillator $(S^2
    \times \R^2,\, \omega_{S^2}\oplus \omega_0,\, (J,\,H))$ is a
    completely integrable system, meaning that the Poisson bracket
    $\{J,\,H\}$ vanishes everywhere\footnote{equivalently the
      Hamiltonian vector field $\mathcal{X}_J$ is constant along the
      flow of $\mathcal{X}_H$}.
    
    In addition, the map $J$ is the momentum map for the Hamiltonian
    circle action of $S^1$ on $S^2 \times \mathbb{R}^2$ that rotates
    simultaneously horizontally about the vertical axes on $S^2$, and
    about the origin on $\R^2$.

    The singularities of the coupled spin--oscillator are
    non\--degenerate and of elliptic\--elliptic,
    transversally\--elliptic or focus\--focus type. It has exactly one
    focus\--focus singularity at the ``North Pole''
    $((0,\,0,\,1),\,(0,\,0)) \in S^2 \times \R^2$ and one
    elliptic\--elliptic singularity at the ``South Pole''
    $((0,\,0,\,-1),\,(0,\,0))$.
  \end{prop}

\begin{cor}
  The coupled spin--oscillator $(S^2 \times \R^2,\, \omega_{S^2}\oplus
  \omega_0,\, (J,\,H))$ is a semitoric integrable system.
\end{cor}

  Computing the Taylor series invariant at the focus\--focus singularity
  is rather involved.  At this point we are able to compute the first two
  terms $a_1,\,a_2$ (for the coupled spin\--oscillators). Even in this
  case one has to do a delicate
  coordinate analysis of flows involving Eliasson's coordinates, and
  the computation of various integrals.

\begin{figure}[htbp]
  \begin{center}
    \includegraphics[width=5.3cm]{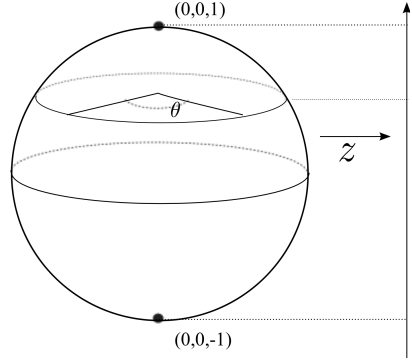}
    \caption{Spin model with momentum map $z$. Here
    $(\theta,\, z)$ are the angle\--height coordinates on the unit sphere $S^2$.}
    \label{fig:fig1}
  \end{center}
\end{figure}

  \subsection{Set up for coupled spin\--oscillators --- Integral
    formulas for singularity invariant}

  Throughout we let $M=S^2 \times \R^2$ and $F=(J,\,H)$. 
  In this set up stage we introduce the $1$\--forms $\kappa_{1,c}$ and
  $\kappa_{2,c}$ in terms of which the Taylor series in defined in
  \cite{san-semi-global}, and we recall limit integral formulas for
  the Taylor series invariant. Then we introduce the limit theorem
  proved in the semiclassical paper \cite[Proposition 6.8]{san-focus},
  which will be the key ingredient for the computation .

  The formulas that we present here do not correspond to the exact
  statements in the corresponding papers, but can be immediately
  deduced from it assuming the context of the present paper.

\paragraph{The one forms $\kappa_{1,c}$ and $\kappa_{2,c}$.}
As usual, we denote by $\mathcal{X}_{q_i}$ the Hamiltonian vector
field generated by $q_i$, $i=1,\,2$.  Let $c$ be a fixed regular value
of $F$.  Let $\kappa_{1,c} \in \Omega^1(\Lambda_c),\, \kappa_{1,c} \in
\Omega^1(\Lambda_c)$ be the smooth $1$\--forms on the fiber
$\Lambda_c:=F^{-1}(c)$ corresponding to the value $c$ defined by the
conditions
\begin{eqnarray} \label{for:kappa1}
  \kappa_{1,c}(\mathcal{X}_{q_1}):=-1,\,\,\,\,
  \kappa_{1,c}(\mathcal{X}_{q_2}):=0,
\end{eqnarray}
and
\begin{eqnarray} \label{for:kappa2}
  \kappa_{2,c}(\mathcal{X}_{q_1}):=0,\,\,\,\,
  \kappa_{2,c}(\mathcal{X}_{q_2}):=-1.
\end{eqnarray}
Note that the conditions in (\ref{for:kappa1}) and (\ref{for:kappa2})
are enough to determine $\kappa_{1,c}$ and $\kappa_{2,c}$ on
$\Lambda_c$ because $\mathcal{X}_{q_1},\, \mathcal{X}_{q_2}$ form a
basis of each tangent space.

We will call $\kappa_{1,0}$, $\kappa_{2,0}$ the corresponding form
defined in the same way as $\kappa_{1,c}$, $\kappa_{2,c}$, but only on
$\Lambda_0\setminus \{m\}$, where $m=(0,0,1,0,0)$ is the singular
point of the focus\--focus singular fiber $\Lambda_0$.

\begin{remark}
  The forms $\kappa_{1,c},\, \kappa_{2,c}$ , $i=1,\,2$ are closed. See
  also \cite[Section 3.2.1]{san-focus}.
\end{remark}

\paragraph{Limit integral formula for Taylor invariants.}

The following result will be key for our purposes in the present
paper.

\begin{lemma} \label{lem:limit} Let $(S) \in \mathbb{R}[[X,\,Y]]$ be
  the Taylor series invariant of the coupled\--spin oscillator. Then
  the first terms of the Taylor series are given by the limits of
  integrals $a_1=\lim_{c \to 0} \Big(
  \int_{\gamma_c}\kappa_{1,c}+\op{arg}(c) \Big) $ and $a_2=\lim_{c \to
    0} \Big( \int_{\gamma_c}\kappa_{2,c}+\op{ln} \vert c \vert \Big)$.
\end{lemma}
 
\begin{proof}
  It follows from the definition of the dynamical invariants
  $\tau_1(c)$ and $\tau_2(c)$ in Section \ref{taylor:sec} and the
  definition of $\kappa_{1,c}$ and $\kappa_{2,c}$ in
  (\ref{for:kappa1}) and (\ref{for:kappa2}) respectively that $
  \tau_i(c)=\int_{\gamma_c} \kappa_{i,c},\,\,\, i=1,\,2.  $ The first
  two terms of the Taylor series invariant $\sigma_1(0)$ and
  $\sigma_2(0)$ where $ \sigma_1=\tau_1+\op{arg}(c)$ and $
  \sigma_2=\tau_2-\op{ln}\vert c \vert $.

  Since $\sigma_1$ and $\sigma_2$ are smooth, we have that $
  a_1=\sigma_1(0)=\lim_{c \to 0} \Big(
  \int_{\gamma_c}\kappa_{1,c}+\op{arg}(c) \Big) $ and $
  a_2=\sigma_2(0)=\lim_{c \to 0} \Big(
  \int_{\gamma_c}\kappa_{2,c}+\op{ln} \vert c \vert \Big) $.
\end{proof}

\paragraph{Localization on the critical fiber.}

On the other hand, we have the following \cite[Proposition
6.8]{san-focus} result proved by the second author.

\begin{theorem}[\cite{san-focus}] \label{san's} Let $\gamma_0$ be a
  radial simple loop.  The integrals in Lemma \ref{lem:limit} are
  respectively equal to
  \begin{eqnarray}\label{equ:formula2}
    a_1=\lim_{c \to 0} \Big(   \int_{\gamma_c}\kappa_{1,c}+\op{arg}(c)      \Big)=
    \lim_{(s,\,t) \to (0,\,0)} \Big(\int_{A_0=\gamma_0(s)}^{B_0=\gamma_0(1-t)}
    \kappa_{1,0}+(t_A-\theta_B)\Big),
  \end{eqnarray}
  and
  \begin{eqnarray} \label{ex1} a_2:=\lim_{c \to 0} \Big(
    \int_{\gamma_c}\kappa_{2,c}+\op{ln} \vert c \vert \Big)=
    \lim_{(s,\,t) \to (0,\,0)}
    \Big(\int_{A_0:=\gamma_0(s)}^{B_0(t):=\gamma_0(1-t)}
    \kappa_{2,0}+\op{ln}(r_{A_0}\rho_{B_0}) \Big),
  \end{eqnarray}
  where for any point $A$ in $M$ close to $m$ with Eliasson
  coordinates $(x_1,\,x_2,\, \xi_1,\, \xi_2)$ as defined in equation
  (\ref{equ:cartan}), we denote by $(r_A,\,t_A,\, \rho_A, \,
  \theta_A)$ the polar symplectic coordinates\footnote{These
    coordinates $(r_A,\,t_A,\, \rho_A, \, \theta_A)$ should not be
    confused with the coordinates $(r,\,t,\, \rho, \, \theta)$ without
    the subscript, which are coordinates in $\R^2 \times S^2$.}  of
  $A$, i.e. $(r_A,\, t_A)$ are polar coordinates corresponding to
  $(x_1,\,x_2)$ and $(\rho_A, \, \theta_A)$ are polar coordinates
  corresponding to $(\xi_1,\, \xi_2)$.
\end{theorem}

\subsection{Computation of integral limit formulas for coupled
  spin\--oscillators}

Now, in order to apply Theorem \ref{san's} we need to find the curve
$\gamma_0$, as well as the $1$\--form $\kappa$ and the coordinates
$(r,\,\theta,\, \rho, \, \alpha)$, both of which are defined on
$\Lambda_0$.  First we describe a parametrization of $\Lambda_0$, and
then we use this parametrization to define $\gamma_0$. We have divided
the computation into five steps.

\subsection*{Stage 1 --- Eliasson's coordinates $(x_1,\,
  x_2,\,\xi_1,\, \xi_2)$}

We find explicitly symplectic coordinates $(\hat{x}_1,\, \hat{x}_2,\,
\hat{\xi}_1,\, \hat{\xi}_2) \in M=S^2 \times \R^2$ in which the
``momentum map'' $F \colon M \to \R^2$ for the coupled
spin\--oscillator has the form (\ref{equ:cartan}), up to a third order
approximation, i.e. up to $(\mathcal{O}(\hat{x}_1,\, \hat{x}_2,\,
\hat{\xi}_1,\,\hat{\xi}_2))^3$.  For brevity write $\mathcal{O}(3)=
(\mathcal{O}(\hat{x}_1,\,\hat{x}_2,\, \hat{\xi}_1,\, \hat{\xi}_2))^3$.

\begin{lemma}
  \label{lem:El}
  Consider the map $ \hat{\phi} \colon \op{T}_{(0,\, 0, \, 0, \, 0)}
  \mathbb{R}^4 \to \op{T}_{(0,\,0,\,1,\, 0,\,0)}(S^2 \times \R^2) $
  given by
$$
\phi(\hat{x}_1,\, \hat{x}_2,\, \hat{\xi}_1,\, \hat{\xi}_2)
=(v:=\frac{1}{\sqrt{2}}(\hat{x}_2+\hat{\xi}_1),
\,\,x:=\frac{1}{\sqrt{2}} (\hat{x}_2-\hat{\xi}_1),\,
u:=\frac{1}{\sqrt{2}} (-\hat{x}_1+\hat{\xi}_2),
\,\,y:=\frac{1}{\sqrt{2}}(\hat{x}_1+\hat{\xi}_2)).
$$
The map $\hat{\phi}$ is a linear symplectomorphism, i.e. an
automorphism such that $\phi^*\Omega{}=\omega_0$, where $
\omega_0=\op{d}\!\hat{x}_1\wedge \op{d}\!\hat{\xi}_1 \oplus
\op{d}\!\hat{x}_2\wedge\op{d}\!\hat{\xi}_2 $ is the standard
symplectic form on $\mathbb{R}^4$, and
$\Omega=(\omega_{S^2}\oplus\op{d}\!u \wedge
\op{d}\!v)_{\upharpoonright T_{(0,0,1,0, 0)} (S^2\times\R^2)} $
(recall $\omega_{S^2}$ is the standard symplectic form on $S^2$).  In
addition, $\hat{\phi}$ satisfies the equation
$\op{Hess}(\widetilde{F}) \circ \hat{\phi} = (q_1,\, q_2)$, where $
\widetilde{F}:=B\circ (F-F(m))=B \circ (F-(1,\,0)) \, \colon M \to
\R^2, $ for the matrix $ B:=\left( \begin{array}{cc}
    1 & 0\\
    0 & 2
  \end{array} \right).
$
\end{lemma}

In the above statement, we identify a Hessian with its associated
quadratic form on the tangent space.

\subsection*{Stage 2 --- Curve and Singular Fiber Parametrization}

\paragraph{Parametrization of $\Lambda_0$.}
Let's now parametrize the singular fiber $\Lambda_0:=F^{-1}(1,\,0)$,
where $F=(J,\,H)$ as usual.  This singular fiber $\Lambda_0$
corresponds to the system of equations $J=1$ and $H=0$, which
explicitly is given by system of two nonlinear equations
$J=(u^2+v^2)/2 + z=0$ and $H= \frac{1}{2} (ux+vy)=0$.  on the
coordinates $(x,\,y,\,z,\, u,\,v)$ on the coupled spin oscillator
$M=S^2 \times \mathbb{R}^2$.

In order to solve this system of equations we introduce polar
coordinates $u+\ii v=r\op{e}^{\ii t}$ and $ x+\ii y=\rho\op{e}^{\ii
  \theta}$ where recall that the $2$\--sphere $S^2 \subset
\mathbb{R}^3$ is equipped with coordinates $(x,\,y,\,z)$, and $\R^2$
is equipped with coordinates $(u,\, v)$.
  
For $\epsilon=\pm 1$, we consider the mapping $ S_\epsilon : [-1,\,
1]\times \R/2\pi\Z \to \R^2\times S^2 $ given by the formula $
S_\epsilon(p) = (r(p)\,\op{e}^{\ii t(p)}, \,(\rho(p)\,\op{e}^{\ii
  \theta(p),\, z(p)})) $ where $p=(\tilde{z}, \, \tilde{\theta})\in
[-1,1]\times[0,2\pi)$ and
\begin{equation} \nonumber
  \begin{cases}
    r(p) = \sqrt{2(1-\tilde{z})}\\
    t(p) = \tilde{\theta} + \epsilon\frac{\pi}{2}\\
    \rho(p)= \sqrt{1-\tilde{z}^2}\\
    \theta(p) = \tilde{\theta}\\
    z(p) = \tilde{z}.
  \end{cases}
\end{equation}

\begin{prop}
  The map $S_{\epsilon}$, where $\epsilon=\pm1$, is continuous and
  $S_{\epsilon}$ restricted to $(-1,\,1) \times \R/2\pi\Z $ is a
  diffeomorphism onto its image.  If we let
  $\Lambda_0^{\epsilon}:=S_{\epsilon}([-1,\, 1]\times \R/2\pi\Z )$,
  then $\Lambda_0^1\cup \Lambda_0^2=\Lambda_0$ and
$$\Lambda_0^1 \cap \Lambda_0^2=\Big( \{(0,\,0)\} \times \{(1,\,0,\,0)\} \Big)\cup 
\Big(C_2 \times \{(0,\,0,\,-1)\} \Big),$$ where $C_2$ denotes the
circle of radius $2$ centered at $(0,\,0)$ in $\R^2$.  Moreover,
$S_{\epsilon}$ restricted to $(-1,\,1) \times \R/2\pi\Z $ is a smooth
Lagrangian embedding into $\R^2 \times S^2$.
\end{prop}

\begin{proof}
  On the one hand we have that $z^2=1-x^2-y^2=1-\rho^2$.  The
  expressions for the maps $J$ and $H$ in the new coordinates $(r, \,
  t,\, \rho, \, \theta)$ are
  \begin{eqnarray}
    J=\frac{1}{2}r^2 \pm \sqrt{1-\rho^2},\,\,\,\,\,\,\,\,\,\,\,\,\,\,\,\,\,\,\,\,\,\,\,\,\,\,
    H=\frac{r \rho}{2} \, \op{cos}(t-\theta). \label{JH:for}
  \end{eqnarray}

  In virtue of the formula for $H$ in the right hand\--side of
  (\ref{JH:for}), if $H=0$ then $r=0$ or $\rho=0$ or
  $t-\theta=\frac{\pi}{2} (\op{mod} \pi)$, which leads to three
  separate cases. The first case is when $r=0$; then $ J=\pm
  \sqrt{1-\rho^2}=1, $ and hence $\rho=0$.  Hence the only solution is
  $(u,\,v,\,x,\,y,\,z)=(0,\,0,\,0,\,0,\,1)$. The second case is when
  $\rho=0$; then either $z=1$ and $r=0$, or $z=-1$ and $r=2$. Hence
  the set of solutions consists of $(0,\,0,\,0,\,0,\,1)$ and the
  circle $r=2$, $\rho=0$ and $z=-1$.  Finally, the third case is when
  $t-\theta=\frac{\pi}{2}(\op{mod} \pi)$; because $J=1$ and $H=0$, it
  follows from the formula for $z$ above and the left hand\--side of
  (\ref{JH:for}) that $ r^2=2(1-z).  $ Hence the set of solutions
  $\Lambda_0$ is equal to the set of points $(r\op{e}^{\ii t},\,
  \rho\, \op{e}^{\ii \theta})$ such that
  \begin{equation}
    \label{eq:1}
    \begin{cases}
      r=\sqrt{2(1-z)},\,\,\,\, z \in [-1,\,1]\\
      \theta=t-\frac{\pi}{2} \,\,\,\,\, \textup{or}\,\,\,\,\, \theta=t+\frac{\pi}{2}, \,\,\,\, t \in [0,\, 2\pi) \\
      \rho=\sqrt{1-z^2}
    \end{cases}
  \end{equation}
  This case contains the previous two cases, which proves statement
  (3) part (i) in virtue of expression (\ref{ex1}).  The other
  statements are left to the reader.
\end{proof}

\begin{remark}
  The singular fiber $\Lambda_0$ consists of two sheets glued along a
  point and a circle; topologically $\Lambda_0$ is a pinched torus,
  i.e.  a $2$\--dimensional torus $S^1 \times S^1$ in which one circle
  $\{p\} \times S^1$ is contracted to a point (which is of course not
  a a smooth manifold at the point which comes from the contracting
  circle).
\end{remark}

\paragraph{The radial vector field $\mathcal{X_H}$ on $\Lambda_0$.}

\begin{prop} \label{lem:key} Let $\mathcal{X}_{q_i}$ be the
  Hamiltonian vector field of $q_i$ (which recall is defined in
  saturated neighborhood of the singular fiber $\Lambda_0$).  On the
  singular fiber $\Lambda_0$, the vector fields $\mathcal{X}_{q_1},\,
  \mathcal{X}_J$ and $\mathcal{X}_{q_2},\, \mathcal{X}_H$ are linearly
  independent, precisely: $ \mathcal{X}_{q_1}=\mathcal{X}_J,
  \,\,\,\,\, \mathcal{X}_{q_2}=2\, \mathcal{X}_H.  $ In particular the
  vector field $\mathcal{X}_H$ is radial.
\end{prop}

\begin{proof}
  It follows from Eliasson's theorem that there exists a smooth
  function $h$ such that $ q=h \circ F $ and $\op{d}\!h(0)$ is the
  invertible $2$ by $2$ matrix $B$ in Lemma \ref{lem:El}.

  Then on $\Lambda_0$ we have that
  \begin{eqnarray} \label{1} \mathcal{X}_{q_i}=\frac{\partial
      h_i}{\partial J}\, \mathcal{X}_J +\frac{\partial h_i}{\partial
      H} \, \mathcal{X}_H, \,\,\,\,\,\,\, i=1,\,2.
  \end{eqnarray}

  Because the coefficients are constant along $\Lambda_0$, it is
  sufficient to do the computation at the origin. At the origin the
  computation is given by the matrix $B$ in Lemma \ref{lem:El}, so we
  have that $\frac{\partial h_1}{\partial J}(0)=,\, \frac{\partial
    h_1}{\partial H}(0)=0, \, \frac{\partial h_2}{\partial J}(0)=0$
  and $\frac{\partial h_2}{\partial H}(0)=2.$ The proposition follows
  from (\ref{1}).
\end{proof}

In the following section we will need to use explicitly the
Hamiltonian vector field $\mathcal{X}_H$, and therein it will be most
useful to a have the following explicit coordinate expression.

\begin{lemma} \label{lem:xh} The Hamiltonian vector field
  $\mathcal{X}_H$ of $H$ is of the form
$$
\mathcal{X}_H=\frac{y}{2} \frac{\partial}{\partial u} -\frac{x}{2}
\frac{\partial}{\partial v} +\frac{-yu+xv}{2}\frac{\partial }{\partial
  z} -\frac{z(xu+yv)}{2(1-z^2)} \frac{\partial}{\partial \theta}.
 $$
\end{lemma}

\begin{proof}
  For this computation let us use coordinates $(u,\, v,\, z,\,
  \theta)$ as a parametrization of $\R^2 \times S^2$.

  The coordinate expression for the Hamiltonian $H$ is $
  H=\frac{1}{2}(xu+yv)=\frac{1}{2}(\rho \cos \theta u +\rho \sin
  \theta v), $ Then the Hamiltonian vector field $\mathcal{X}_H$ is of
  the form $ \mathcal{X}_H=a \frac{\partial}{\partial u} +b
  \frac{\partial}{\partial v} +c \frac{\partial}{\partial z} +d
  \frac{\partial}{\partial \theta},$ where since the symplectic form
  on $\R^2 \times S^2$ in these coordinates is $\op{d}\!u \wedge
  \op{d}\!v +\op{d}\!\theta\wedge \op{d}\!z$, the function coefficient
  $a$ (which will be important later in the proof) is given by
  \begin{eqnarray} \label{for:a} a= \frac{\partial H}{\partial
      v}=\frac{1}{2} \rho \sin(\theta)=\frac{y}{2}
  \end{eqnarray}
  and the other function coefficients are given by $b= -
  \frac{\partial H}{\partial u}= \rho \cos(\theta)=-\frac{x}{2},
  c=\frac{\partial H}{\partial \theta}=\frac{\rho}{2}
  (-\sin(\theta)u+\cos(\theta)v)= \frac{-yu+xv}{2} $ and $d=
  -\frac{\partial H}{\partial z}$.
  
  We need to compute $d$ explicitly.  Since $\frac{\partial
    \theta}{\partial z}=0$ because the angle $\theta$ does not depend
  on the height $z$, and $
  \frac{\op{d}\!\rho}{\op{d}\!z}=-\frac{z}{\sqrt{1-z^2}},$ we have
  that
  \begin{eqnarray} \label{for:dd} \frac{\partial x}{\partial
      z}=\frac{\partial x}{\partial \rho}\frac{\partial \rho}{\partial
      z} +\frac{\partial x}{\partial \theta}\frac{\partial
      \theta}{\partial z}
    =\frac{\partial x}{\partial \rho}\frac{\partial \rho}{\partial z}=\frac{-xz}{\rho^2} \\
    \frac{\partial y}{\partial z}=\frac{\partial y}{\partial
      \rho}\frac{\partial \rho}{\partial z} +\frac{\partial
      y}{\partial \theta}\frac{\partial \theta}{\partial z}
    =\frac{\partial y}{\partial \rho}\frac{\partial \rho}{\partial
      z}=\frac{-yz}{\rho^2} \label{for:dd2}
  \end{eqnarray}

  It follows that from (\ref{for:dd}) and (\ref{for:dd2}) that the
  function coefficient $d$ is given by
$$
d= -\frac{\partial H}{\partial z}= - \frac{\partial H}{\partial x}
\frac{\partial x}{\partial z} +\frac{\partial H}{\partial y}
\frac{\partial y}{\partial z}= \frac{u}{2} \,
\frac{-xz}{\rho^2}+\frac{v}{2}\, \frac{-yz}{2\rho^2}
=-\frac{z(xu+yv)}{2\rho^2}=-\frac{z(xu+yv)}{2(1-z^2)}.
$$
\end{proof}

\paragraph{Definition of a simple ``radial'' loop in $\Lambda_0$.}

In order to apply the theorem it is enough to take $\gamma_0$ to be an
integral curve of the radial vector field $\mathcal{X}_H$.

We define $\gamma_0$ as the simple loop obtained through the
parametrizations $S_+$ and $S_-$ by letting $\tilde{z}$ run from $-1$
to $1$ and back to $-1$, respectively. For instance, one can use the
formula

\begin{eqnarray}
  \gamma_0(s)&:=&\left\{\begin{array}{rl}
      S_1(-1+4s,\, -\frac{\pi}{2}) & \textup{ if } 0\le s \le \frac{1}{2}; \nonumber \\
      S_2(3-4s,\, \frac{\pi}{2}) & \textup{ if } \frac{1}{2}<s \le 1 . 
    \end{array} \right. 
\end{eqnarray}

\begin{cor} \label{cor:xh} Along the curve $\gamma_0$ we have
  \begin{eqnarray} \label{for:xh}
    \mathcal{X}_H\Big|_{\gamma_0}=\frac{y}{2} \frac{\partial}{\partial
      u} -\frac{yu}{2} \frac{\partial}{\partial z}.
  \end{eqnarray}
\end{cor}

\begin{proof}
  We use the notation of Lemma \ref{lem:xh}.  Along $\gamma_0$ we have
  $v=0$, $x=0$ and $\theta=\pi$ or $\theta=\frac{3\pi}{2}$.  Hence $
  a= \frac{y}{2}, b= 0,, c= - \frac{yu}{2}, d=0$.  Therefore the
  vector field $\mathcal{X}_H$ along the curve $\gamma_0$ is given by
  (\ref{for:xh}).
\end{proof}

Using Corollary \ref{cor:xh} we describe the very explicit relation
between the curve $\gamma_0$ and the Hamiltonian vector field
$\mathcal{X}_H$.

\begin{prop}
  The curve $\gamma_0 \colon [0,\,1] \to M$ is an integral curve of
  $\mathcal{X}_H$.
\end{prop}

\begin{proof}
  Since by construction the vector field $S_*(\frac{\partial}{\partial
    \tilde{z}})$ is tangent to the curve $\gamma_0$, it is enough to
  show that $S_*(\frac{\partial}{\partial \tilde{z}})$ is colinear to
  $\mathcal{X}_H$ are colinear at each point.

  A computation gives that
  \begin{eqnarray} \label{ex'} S_*\Big(\frac{\partial}{\partial
      \tilde{z}}\Big)=\frac{\partial}{\partial z}
    -\frac{1}{\sqrt{2(1-z)}} \, \frac{\partial}{\partial
      r}+\frac{z}{\sqrt{1-z^2}} \, \frac{\partial}{\partial \rho}.
  \end{eqnarray}

  On the other hand
  \begin{eqnarray} \label{for:u} u=\sqrt{2(1-z}),
  \end{eqnarray}
  and since $(r,\,t)$ are polar coordinates for $(u,\,v)$, $
  \frac{\partial}{\partial r}=\cos t \, \frac{\partial}{\partial u}
  +\sin t \, \frac{\partial}{\partial v}, $ which at $t=0$ gives that
  $\frac{\partial}{\partial r}= \frac{\partial}{\partial u}$.
  Therefore, because at $t=0$ the last factor of (\ref{ex'}) is zero,
  we conclude from (\ref{for:u}) that
  \begin{eqnarray}
    S_*\Big(\frac{\partial}{\partial \tilde{z}}\Big)=\frac{\partial}{\partial z}
    -\frac{1}{u}\, \frac{\partial}{\partial u}.
  \end{eqnarray}.

  It follows from (\ref{for:xh}) that $ \mathcal{X}_H=-\frac{yu}{2} \,
  S_*\Big(\frac{\partial}{\partial \tilde{z}}\Big), $ which shows that
  $\mathcal{X}_H$ and $S_*(\frac{\partial}{\partial \tilde{z}})$ are
  colinear at every point, as desired.

\end{proof}

\subsection*{Stage 3 --- Integration in linearized Eliasson's
  coordinates}

Let $\phi$ be a local symplectic map such that $ g\circ F\circ \phi =
q \text{ on } \R^4$, as given by Eliasson's normal form theorem.  The
integrals in (..) are defined in terms of the corresponding canonical
coordinates $(x_1,\, x_2,\, \xi_1,\, \xi_2)$ in $\R^4$.

Because our computation is local, we can use instead the linearized
coordinates that we have defined in Lemma \ref{lem:El}.  More
precisely, one can always choose $\phi$ such that the tangent map $
\op{d}_{(0,\,0,\,0,\,0)} \phi: \op{T}_{(0,\,0,\,0,\,0)} \R^4 \to
\op{T}_{(0,\,0,\,1,\,0,\,0)} S^2 \times \mathbb{R}^2 $ is equal to
$\hat{\phi}$, and this gives local coordinates
$(\hat{x}_1,\hat{x}_2,\hat{\xi}_1, \hat{\xi}_2)$ in a neighborhood of
$m$, such that $ B\circ F (\hat{x}_1,\hat{x}_2,\hat{\xi}_1,
\hat{\xi}_2) = q( \hat{x}_1,\hat{x}_2,\hat{\xi}_1, \hat{\xi}_2) +
\mathcal{O}(3).  $

Note that these coordinates are not symplectic, except at $m$.

\begin{lemma}
  The integral (\ref{ex1}) gives us the same result when computed in
  linearized coordinates, i.e. upon replacing $r_A$ by $\hat{r}_A$,
  $t_A$ by $\hat{t}_A$, $\rho_A$ by $\hat{\rho}_A$ and $\theta_A$ by
  $\hat{\theta}_A$.
\end{lemma}

\begin{proof}
  Since $r_A^2=x_1^2+x_2^2$, then
  \begin{eqnarray}
    \hat{r}_A^2=\hat{x}_1^2+\hat{x}_2^2=x_1^2+x_2^2+\mathcal{O}(3)
    =r_A^2+\mathcal{O}(3)
    \label{rA:for}
  \end{eqnarray}
  We know that
  $\frac{\mathcal{O}(3)}{x_1^2+x_2^2}=\mathcal{O}(1)$,
  and therefore it follows from (\ref{rA:for}) that
  \begin{eqnarray} \label{1:for}
    \op{ln}(\hat{r}_A^2)=\op{ln}(r_A^2+\mathcal{O}(3))
    =\op{ln}\Big(1+\frac{\mathcal{O}(3)}{r_A^2} \Big) +\op{ln}(r_A^2)
    =\op{ln}(1+\mathcal{O}(1))+\op{ln}(r_A^2)
    =\mathcal{O}(1)+\op{ln}(r_A^2).
  \end{eqnarray}
  Similarly
  $\op{ln}(\hat{\rho}_B^2)=\mathcal{O}(1)+\op{ln}(\rho_B^2)$.  Hence $
  \op{ln}(r_A
  \rho_B)=\op{ln}(r_A)+\op{ln}(\rho_B)=\op{ln}(\hat{r}_A)+\op{ln}(\hat{\rho}_B)
  =\op{ln}(\hat{r}_A\hat{\rho}_B)+\mathcal{O}(1).  $ Then
  \begin{eqnarray} \label{for:new} \lim_{(s,\,t) \to (0,\,0)}
    \op{ln}(r_{A_0} \rho_{B_0}) -
    \op{ln}(\hat{r}_{A_0}\hat{\rho}_{B_0})=0.
  \end{eqnarray}
  It follows from expressions (\ref{ex1}) and (\ref{for:new}) that
  \begin{eqnarray} \label{for:key} a_2 = \lim_{(s_A,\,s_B) \to
      (0,\,0)} \Big( \int_{A_0=\gamma_0(s_A)}^{B_0=\gamma_0(1-s_B)}
    \kappa_{2,0}+\op{ln}\vert \hat{r}_{A_0} \hat{\rho_{B_0}}\vert
    \Big). \label{ex2}
  \end{eqnarray}
  This concludes the proof.
\end{proof}

\subsection*{Stage 4 --- Computation of the first order Taylor series
  invariants $a_1$ and $a_2$}

In order to compute the integrals in (\ref{for:key}) we can replace
$\gamma_0$ by any integral curve of $\mathcal{X}_H$ with the same
endpoints.  Thus, let $\gamma$ be a solution to $
\dot{\gamma}=\mathcal{X}_H \circ \gamma.  $ By definition, for any
1-form $\kappa$,
\begin{eqnarray} \label{for:integral} \int_{A_0:=\gamma(s_1), \,\,
    \textup{along}\, \gamma}^{B_0:=\gamma(s_2)}
  \kappa=\int_{s_1}^{s_2}
  \kappa_{{\gamma(s)}}(\dot{\gamma}(s))\op{d}\!s=
  \int_{s_1}^{s_2}\kappa_{{\gamma(s)}}(\mathcal{X}_H
  (\gamma(s)))\op{d}\!s.
\end{eqnarray}

\begin{theorem} \label{theo:a1} Let $(S) \in \mathbb{R}[[X,\,Y]]$ be
  the Taylor series invariant of the couple\--spin oscillator. Then
  the first coefficient of the first term of the series is given by
  $a_1=\frac{\pi}{2}$.  The second coefficient of the first term of
  the first order Taylor series invariant is $a_2=5 \,\op{ln}2$.
\end{theorem}

\begin{proof}
  We have divided the computation of $a_2$ in several steps.
  \\
  \\
  \emph{\underline{Step 1}: Set-up of the integral of $\kappa_{2,0}$.}
  We need to compute expression (\ref{ex2}).

  Let $a$ be given by (\ref{for:a}).
  
  In view of (\ref{for:xh}), the path $\gamma$ between $A_0$ and $B_0$
  can be parametrized by the variable $u$. This means that the path
  $\gamma$ is obtained by first increasing $u$ up to $u=2$ on the
  first sheet (parametrized by $S_1$) and then decreasing $u$ on the
  second sheet (parametrized by $S_2$).

  By Lemma \ref{lem:key} we know that
  $\mathcal{X}_{q_2}=2\mathcal{X}_H$ and hence
  $(\kappa_{2,0})_{\gamma(s)}(\mathcal{X}_H (\gamma(s)))=
  \frac{(\kappa_{2,0})_{\gamma(t)}(\mathcal{X}_{q_2}(\gamma(s)))}{2}$.
  By definition of $\kappa_{2,0}$ we know that
  $\kappa_{2,0}(\mathcal{X}_{q_2})=-1$ and hence it follows from
  (\ref{for:integral}) that $\int_{A_0,\, \textup{along}\,
    \gamma}^{B_0} \kappa_{2,0}=\int_{s_1}^{s_2}
  \frac{\op{d}\!s}{2}$. Since $\frac{\op{d}\!u}{\op{d}\!s}$ is equal
  to $a=\frac{y}{2}$ we have that
  \begin{eqnarray}
    \int_{A_0, \,\, \textup{along}\, \gamma}^{B_0} \kappa_{2,0}
    = \int_{s_1}^{s_2} \frac{\op{d}\!s}{2}  
    = \int_{u_1}^{2} \frac{\op{d}\!u}{y_+(u)}+\int_{2}^{u_2} \frac{\op{d}\!u}{y_{-}(u)}, 
    \label{sum}
  \end{eqnarray}
  where $y_{\pm}(u)$ is the $y$\--coordinate along the part of the
  curve $\gamma_0$ which corresponds to the parametrization $S_{\pm}$,
  respectively.  Our next goal is to compute expression (\ref{sum}).
  \\
  \\
  \emph{\underline{Step 2}: Computation of expression (\ref{sum})}.
  Now, $ y=\rho \sin(\theta)=\pm \rho.$
  
  Now let us express the dependence of $y$ in $u$ along the path
  $\gamma$.  By the equation $J=\frac{1}{2}(u^2+v^2)+z=-1,$ which is
  always true along the singular fiber, we have that, since $v=0$, $
  \frac{u^2}{2}+z=1, $ or in other words, $ z=1-\frac{u^2}{2}$. It
  follows from this equation that
  \begin{eqnarray}
    y_{\pm} 
    =\pm \rho
    = \pm \sqrt{1-z^2} 
    =\pm \sqrt{1-(1-\frac{u^2}{2})^2} 
    =\pm u \sqrt{1-\frac{u^2}{4}}\,\,\, \, \textup{since}\,\, u>0. \label{eqn:a}
  \end{eqnarray}

  On the other hand, note that the function
  $G(t)=\op{ln}\Big(\frac{1}{\cos t}+\tan t\Big)$ is a primitive of
  the function $g(t)= \frac{1}{\cos t}. $ Then by equation
  (\ref{eqn:a}), using the change of variable $ u/2=\cos t, $ and then
  applying the fundamental theorem of calculus we obtain\footnote{The
    integral is equal to $0$ when $u=2$}
  \begin{eqnarray}
    \int_{u_1}^2 \frac{\op{d}\!u}{y_+}
    = \int_{u_1}^2 \frac{\op{d}\!u}{u \sqrt{1-\frac{u^2}{4}}}  
    = - \Big[ \op{ln} \Big(\frac{1}{\cos t}+\tan t\Big) \Big]^0_{t_1} 
    =- \Big[ \op{ln}\Big(\frac{2}{u}+\frac{2}{u}\sqrt{1-\frac{u^2}{4}}\Big) \Big]^2_{u_1}, 
    \nonumber
  \end{eqnarray}
  and simplifying this expression we then obtain
  \begin{eqnarray}
    \int_{u_1}^2 \frac{\op{d}\!u}{y_+}
    = \op{ln}\Big(\frac{2}{u_1}\Big)+\op{ln}\Big(1+\sqrt{1-\frac{u_1^2}{4}}\Big). \label{eqn:intu12}
  \end{eqnarray}

  The goal of this proof is to compute $a_1$, which by (\ref{ex1}) is
  equal to the limit
  \begin{eqnarray} \label{ex1'} \lim_{(s,\,t) \to (0,\,0)}
    \Big(\int_{A_0:=\gamma_0(s)}^{B_0(t):=\gamma_0(1-t)}
    \kappa_{2,0}+\op{ln}(r_{A_0}\rho_{B_0}) \Big), \nonumber
  \end{eqnarray}
  and precisely because this limit exists, we may calculate it along
  the diagonal values given by $u=u_1=u_2$. Then it follows from
  equation (\ref{eqn:intu12}) that
  \begin{eqnarray}
    \int_{A_0}^{B_0} \kappa&=& \int_{u_1}^2 \frac{\op{d}\!u}{y_{+}}
    +\int_2^{u_2}\frac{\op{d}\!u}{y_{-}}
    =2\int_u^2\frac{\op{d}\! y_{+}}{y_{+}}  
    = 2 \Big(\op{ln}\Big(\frac{2}{u}\Big)+\op{ln}\Big(1+\sqrt{1-\frac{u^2}{4}}\Big) \Big). \label{kappa:for}
  \end{eqnarray}
  This concludes this step.
  \\
  \\
  \emph{\underline{Step 3}: Computation of the logarithm factor $
    \op{ln}(\hat{r}_{A_0}\hat{\rho}_{B_0})$}.
  
  From the notation of Stage 1 we have that $
  \hat{r}_A^2=\hat{x}_1^2+\hat{x}_2^2 $ and that $
  \hat{\rho}_A^2=\hat{\xi}_1^2+\hat{\xi}_2^2.  $ Using Lemma
  \ref{lem:El} we find that $
  \hat{r}^2_A=\frac{1}{2}(x^2+y^2+u^2+v^2)+(-uy+vx) $ and $
  \hat{\rho}^2_A=\frac{1}{2}(x^2+y^2+u^2+v^2)+(uy-vx)$.

  We need to compute $\hat{r}_{A_0}$ and $\hat{\rho}_{B_0}$.  The
  points $A_0$ and $B_0$ are in the path $\gamma_0$ and
  $ A_0:=(u_{A_0},\, v_{A_0},\, \theta_{A_0},\, z_{A_0})=(u_{A_0},\,
  0,\, \frac{\pi}{2},\, 1-\frac{u_{A_0}^2}{2}), $ and $
  B_0:=(u_{B_0},\, v_{B_0},\, \theta_{B_0},\, z_{B_0})= (u_{A_0},\,
  0,\, \frac{-\pi}{2},\, 1-\frac{u_{A_0}^2}{2}).  $

  With this information we can compute $\hat{r}_{A_0}$ and
  $\hat{\rho}_{B_0}$ using expression (\ref{eqn:a}) and recalling that
  $x=v=0$ along $\gamma$:
  \begin{eqnarray}
    \hat{r}^2_{A_0}=\frac{1}{2}(u^2-\frac{u^4}{4}+u^2)-u^2\sqrt{1-\frac{u^2}{4}} 
    &=&\frac{u^2}{2}(2-\frac{u^2}{4}-2\sqrt{1-\frac{u^2}{4}}), \label{for:eq}
  \end{eqnarray}
  where here we have also used $
  \rho^2=1-z^2=1-(1-\frac{u^2}{2})^2=u^2-\frac{u^2}{4}.  $ And we also
  have that
  \begin{eqnarray} \label{eqn:rhor}
    \hat{\rho}^2_{B_0}=\hat{r}^2_{A_0}.
  \end{eqnarray}
  It follows from (\ref{for:eq}) and (\ref{eqn:rhor}) that
  \begin{eqnarray}
    \op{ln}(\hat{r}_{A_0}\hat{\rho}_{B_0})
    =\frac{1}{2}\op{ln}(\hat{r}_{A_0}^2\hat{\rho}_{B_0}^2) 
    =\frac{1}{2}\op{ln}(\hat{r}_{A_0}^4) 
    = \op{ln}(\hat{r}^2_{A_0}) 
    = \op{ln} \Big( \frac{u^2}{2}(2-\frac{u^2}{4}+2\sqrt{1-\frac{u^2}{4}})\Big) \nonumber
  \end{eqnarray}
  and therefore that
  \begin{eqnarray}
    \op{ln}(\hat{r}_{A_0}\hat{\rho}_{B_0})
    =2\op{ln}(\frac{u}{\sqrt{2}}) 
    +\op{ln}(2-\frac{u^2}{4}+2\sqrt{1-\frac{u^2}{4}}). \label{log:for}
  \end{eqnarray}

  This concludes the computation of the logarithmic factor.
  \\
  \\
  \emph{\underline{Step 4}: Conclusion}.  It follows from (\ref{ex1}),
  (\ref{kappa:for}) and (\ref{log:for}) that
  \begin{eqnarray}
    a_2 &=&\lim_{u \to 0} \Big( \int_{A_0}^{B_0} \kappa_{2,0}  + \op{ln}(\hat{r}_{A_0}\hat{\rho}_{B_0}) \Big) \nonumber \\
    &=& \lim_{u \to 0} \Big((2 \op{ln}(\frac{2}{u})+2\op{ln}(1+\sqrt{1-\frac{u^2}{4}})
    +2\op{ln}(\frac{u}{\sqrt{2}})+\op{ln}(2-\frac{u^2}{4}+
    2\sqrt{1-\frac{u^2}{4}})
    \Big) \nonumber \\
    &=& 2\op{ln}2+2\op{ln}2-\op{ln}2+2\op{ln}2=5\op{ln}2.
    \label{equ:numerical_formula}
  \end{eqnarray}
  So we have proven that $ a_2=5\op{ln2} $ as we wanted to show.

  In order to find $a_1$, note that the following hold: $ u \ge 0,\,\,
  v=0,\,\, \theta=\frac{\pi}{2} \,\,\textup{or}\,\,
  \frac{3\pi}{2},\,\, \rho=\sqrt{1-z^2},\,\, z=1-\frac{u^2}{4},\,\,
  \rho=\sqrt{u^2-\frac{u^2}{4}}.  $ In this case $x_1=\frac{u\pm
    \rho}{2},\,\,\,\,\,\,x_2=\frac{u\pm \rho}{2},$ and therefore $
  \hat{\theta}=\frac{\pi}{4}.  $ Similarly $\xi_1=\frac{-u \pm
    \rho}{2},\,\,\,\,\,\, \xi_2=\frac{u \mp \rho}{2}=-\xi_1,$ and
  hence $\alpha=\frac{\pi}{4}$.  It follows that
  $\hat{\theta}_{A_0}-\hat{\alpha}_{B_0}=\frac{\pi}{2}$.
  Therefore by Theorem \ref{san's}
$$
a_1= \lim_{(s,\,t) \to (0,\,0)}
\Big(\int_{A_0=\gamma_0(1)}^{B_0=\gamma_0(1-t)}
\kappa_{1,0}+(\hat{\theta}_A-\hat{\alpha}_B)\Big)=\frac{\pi}{2}.
$$
Here we are using that because $\kappa_0(\mathcal{X}_H)=0$ and
$\gamma_0$ is tangent everywhere to $\mathcal{X}_H$ so one has that
$$
\lim_{(s,\,t) \to (0,\,0)}
\Big(\int_{A_0=\gamma_0(1)}^{B_0=\gamma_0(1-t)} \kappa_0\Big)=0.
$$
(See also the paragraphs before Theorem \ref{theo:a1}). This concludes
the proof.
\end{proof}

Theorem \ref{main} follows from Theorem \ref{theo:a1}.

\begin{remark} 
  It is plausible that our proof technique generalizes to compute the
  higher order terms of the Taylor series invariant, but not
  immediately, as we rely on the limit theorem proved in
  \cite{san-focus} which only applies to the first two terms. The
  computation provides more evidence of the fact that from a dynamical
  and geometric view\--point focus\--focus singularities contain a
  large amount of information.
\end{remark}

\section{Convexity theory for coupled spin\--oscillators}
\label{uniqueness:sec}

The plane $\R^2$ is equipped with its standard affine structure with
origin at $(0,0)$, and orientation.  Let
$\textup{Aff}(2,\R^2):=\textup{GL}(2,\R^2)\ltimes\R^2$ be the group of
affine transformations of $\R^2$. Let
$\textup{Aff}(2,\Z):=\textup{GL}(2,\Z)\ltimes\R^2$ be the subgroup of
\emph{integral-affine} transformations.

Let $\mathcal{T}$ be the subgroup of $\op{Aff}(2,\,\Z)$ of those
transformations which leave a vertical line invariant, or
equivalently, an element of $\mathcal{T}$ is a vertical translation
composed with a matrix $T^k$, where $k \in \Z$ and
  $$
  T^k:=\left(
    \begin{array}{cc}
      1 & 0\\ k & 1
    \end{array}
  \right) \in \op{GL}(2,\, \Z).
  $$
  Let $\ell_0\subset\R^2$ be a vertical line in the plane, not
  necessarily through the origin, which splits it into two
  half\--spaces, and let $n\in\Z$. Fix an origin in $\ell$.  Let
  $t^n_{\ell_0} \colon \R^2 \to \R^2$ be the identity on the left
  half\--space, and $T^n$ on the right half\--space. By definition
  $t^n_{\ell_0}$ is piecewise affine.  A \emph{convex polygonal set}
  $\Delta$ is the intersection in $\R^2$ of (finitely or infinitely
  many) closed half\--planes such that on each compact subset of the
  intersection there is at most a finite number of corner points. We
  say that $\Delta$ is \emph{rational} if each edge is directed along
  a vector with rational coefficients.  For brevity, in this paper we
  usually write \emph{``polygon''} instead of \emph{``convex polygonal
    set''}.

  \subsection{Construction of the semitoric polygon invariant}
  \label{semitoric:sec}
  Let $\ell$ be a vertical line through the focus\--focus value $c$.
  Let $B_{\op{r}}:=\op{Int}(B)\setminus \{c\}$, which is precisely the
  set of regular values of $F$.  Given a sign $\epsilon\in\{-1,+1\}$,
  let $\ell^{\epsilon}\subset\ell$ be the vertical half line starting
  at $c$ at extending in the direction of $\epsilon$~: upwards if
  $\epsilon=1$, downwards if $\epsilon=-1$.
 
  In Th.~3.8 in \cite{san-polytope} it was shown that for
  $\epsilon\in\{-1,+1\}$ there exists a homeomorphism $f
  =f_\epsilon\colon B \to \R^2$, modulo a left composition by a
  transformation in $\mathcal{T}$, such that $f|_{(B\setminus
    \ell^{\epsilon})}$ is a diffeomorphism into its image
  $\Delta:=f(B)$, which is a \emph{rational convex polygon},
  $f|_{(B_r\setminus \ell^{\epsilon})}$ is affine (it sends the
  integral affine structure of $B_{\op{r}}$ to the standard structure
  of $\R^2$) and $f$ preserves $J$: i.e.  $
  f(x,\,y)=(x,\,f^{(2)}(x,\,y)).  $ $f$ satisfies further properties
  \cite{san-alvaro-I}, which are relevant for the uniqueness theorem
  proof.  In order to arrive at $\Delta$ one cuts $(J,\,H)(M) \subset
  \R^2$ along the vertical half\--lines $\ell^{\epsilon}$. Then the
  resulting image becomes simply connected and thus there exists a
  global $2$\--torus action on the preimage of this set. The polygon
  $\Delta$ is just the closure of the image of a toric momentum map
  corresponding to this torus action.

    \begin{figure}[h]
      \centering
      \includegraphics[width=0.8\linewidth]{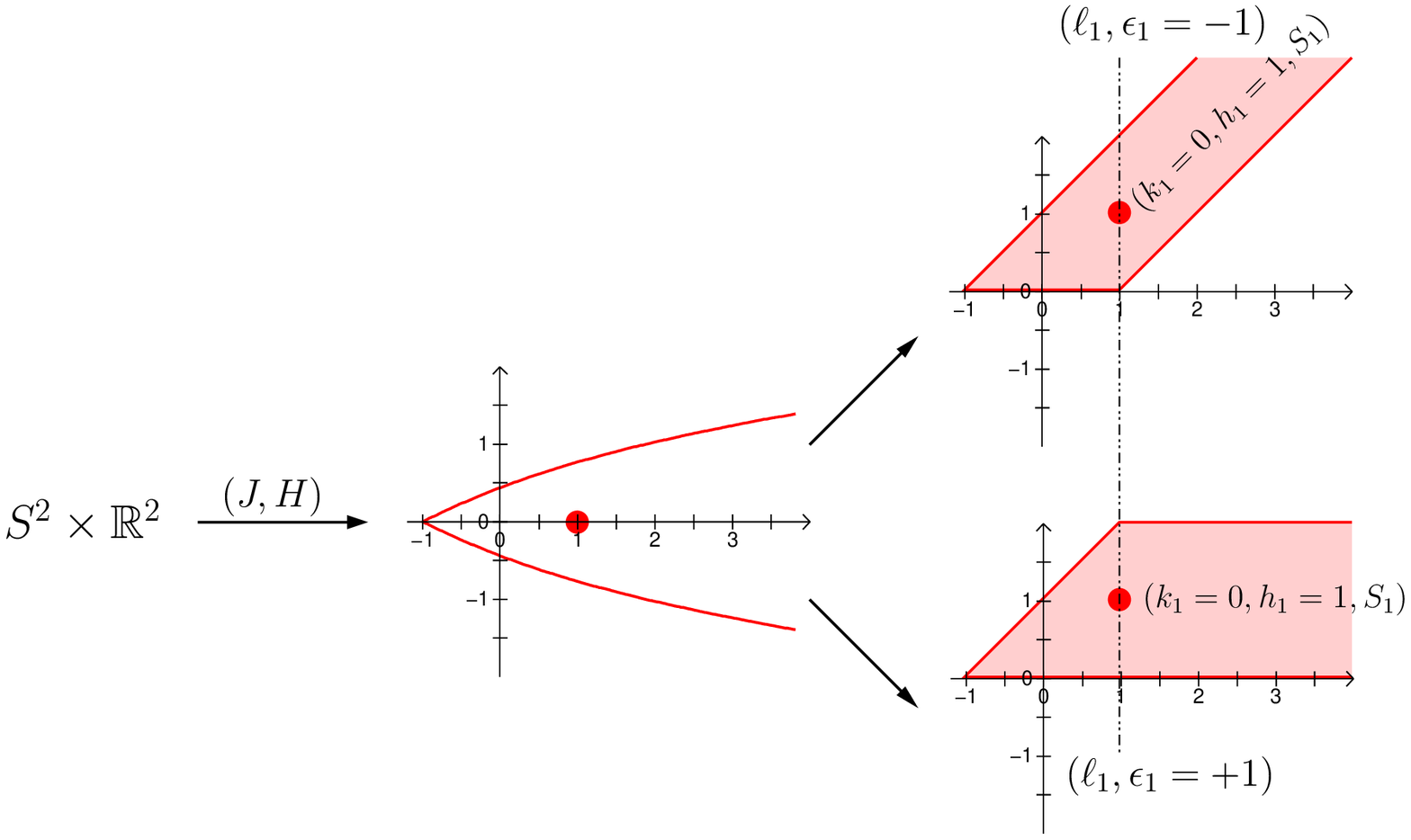}
      \caption{The coupled spin-oscillator example. The middle figure
        shows the image of the initial moment map $F=(J,\, H)$. Its
        boundary is the parametrized curve $(j(s)=\frac{s^2-3}{2s},
        h(s)=\pm\frac{s^2-1}{2s^{3/2}}), \,\,s\in[1,\infty).  $ The
        image is the connected component of the origin. The system is
        a simple semitoric system with one focus\--focus point whose
        image is $(1,\,0)$. The invariants are depicted on the right
        hand-side.  The class of generalized polygons for this system
        consists of two polygons.}
      \label{fig:spin}
    \end{figure}

    We can see that this polygon is not unique. The choice of the
    ``cut direction'' is encoded in the signs $\epsilon$, and there
    remains some freedom for choosing the toric momentum
    map. Precisely, the choices and the corresponding homeomorphisms
    $f$ are the following~:
    \begin{itemize}
    \item[(a)] {\em an initial set of action variables $f_0$ of the
        form $(J,\,K)$} near a regular Liouville torus in \cite[Step
      2, pf. of Th.~3.8]{san-polytope}.  If we choose $f_1$ instead of
      $f_0$, we get a polygon $\Delta'$ obtained by left composition
      with an element of $\mathcal{T}$.  Similarly, if we choose
      $f_1$ instead of $f_0$, we obtain $f$ composed on the left with
      an element of $\mathcal{T}$;
    \item[(b)] {\em an integer $\epsilon \in \{1,\,-1\}$}.  If we
      choose $\epsilon'$ instead of $\epsilon$ we get
      $\Delta'=t_{u}(\Delta) $ with $u=(\epsilon-\epsilon')/2$, by
      \cite[Prop. 4.1, expr. (11)]{san-polytope}.  Similarly instead
      of $f$ we obtain $f'=t_{u} \circ f$.
    \end{itemize}

    Once $f_0$ and $\epsilon$ have been fixed as in (a) and (b),
    respectively, then there exists a unique toric momentum map $\mu$
    on $M_r:=F^{-1}(\textup{Int}{B}\setminus \ell{\epsilon})$ which
    preserves the foliation $\mathcal{F}$, and coincides with
    $f_0\circ F$ where they are both defined. Then, necessarily, the
    first component of $\mu$ is $J$, and we have
    $\overline{\mu(M_r)}=\Delta.$

    We need now for our purposes to formalize choices (a) and (b) in a
    single geometric object.  Let $\op{Polyg}(\R^2)$ be the space of
    rational convex polygons in $\R^2$. Let $\op{Vert}(\R^2)$ be the
    set of vertical lines in $\R^2$.  A {\em weighted polygon} (of
    complexity $1$) is a triple of the form $
    \Delta_{\scriptop{w}}=\Big(\Delta,\, \ell_{\lambda},\,
    \epsilon\Big) $ where $\Delta \in \op{Polyg}(\R^2)$, $\ell \in
    \op{Vert}(\R^2)$, and $\epsilon \in \{-1,\,1\}$. Let
    $G:=\{-1,\,+1\}$. Obviously, the group $\mathcal{T}$ sends a
    rational convex polygon to a rational convex polygon. It
    corresponds to the transformation described in (a). On the other
    hand, the transformation described in (b) can be encoded by the
    group $G$ acting on the triple $\Delta_{\scriptop{w}}$ by the
    formula
$$
\epsilon' \cdot \Big( \Delta,\,\ell_{\lambda} ,\, \epsilon \Big)=
\Big(t_{u}(\Delta),\,\ell_{\lambda},\, \epsilon' \,\epsilon
\Big),\label{equ:actionGs} \nonumber
$$
where $\vec u=(\epsilon-\epsilon')/2$.  This, however, does not always
preserve the convexity of $\Delta$, as is easily seen when $\Delta$ is
the unit square centered at the origin and $\lambda_1=0$. However,
when $\Delta$ comes from the construction described above for a
semitoric system $(J,\,H)$, the convexity is preserved. Thus, we say
that a weighted polygon is \emph{admissible} when the $G$\--action
preserves convexity.  We denote by $\mathcal{W}\op{Polyg}(\R^2)$ the
space of all admissible weighted polygons (of complexity $1$).  The
set $G \times \mathcal{T}$ is an abelian group, with the natural
product action.  The action of $G\times \mathcal{T}$ on
$\mathcal{W}\op{Polyg}(\R^2)$, is given by:
\[
(\epsilon',\,\tau) \cdot \Big( \Delta,\,\ell_{\lambda},\, \epsilon
\Big)= \Big(t_{u}(\tau(\Delta)),\,\ell_{\lambda},\, \epsilon'\,\epsilon
\Big),
\]
where $u=(\epsilon-\epsilon')/2$.  We call a \emph{semitoric polygon}
the equivalence class of an admissible weighted polygon under the $(G
\times \mathcal{T})$\--action.

Let $\Delta$ be a rational convex polygon obtained from the momentum
image $(J,\,H)(M)$ according to the above construction of cutting
along the vertical half-line $\ell^{\epsilon}$.

\begin{definition} \label{generalizedpolytope} The {\em semitoric
    polygon invariant of} $(M,\, \omega,\, (J,\,H))$ is the semitoric
  polygon equal to the $(G \times \mathcal{T})$\--orbit $(G \times
  \mathcal{T})\cdot \Big(\Delta,\, \ell,\, \epsilon \Big) \in
  \mathcal{W}\op{Polyg}(\R^2)/(G \times \mathcal{T}).  $
\end{definition}

\subsection{The semitoric polygon invariant of coupled
  spin\--oscillators}

   \begin{prop}
     The semitoric polygon invariant of the coupled spin\--oscillator
     is the $(G \times \mathcal{T})$\--orbit consisting of the two
     convex polygons depicted on the right hand\--side of Figure
     \ref{fig:spin}.
   \end{prop}
   
   \begin{proof}
     As shown in Figure \ref{fig:spin}, a representative of the
     semitoric polygon invariant is a polygon in $\R^2$ with exactly
     two vertices at $(-1,\,0)$ and $(1,\,0)$, and from these two
     points leave straight lines with slope $1$ (the other possible
     polygon representative has vertices at $(-1,\,0)$ and $(1,\,2)$).
     One finds this polygon simply by combining the information about
     the isotropy weights at the left corner of the polygon (an
     elliptic-elliptic critical value) \cite[Prop.~6.1]{san-polytope},
     together with the formula given in \cite[Thm.~5.3]{san-polytope},
     in which the relation between isotropy weights and the slopes of
     the edges of the polygon is described using the
     Duistermaat\--Heckman function.
   \end{proof}

   \subsection{Classification theory for coupled spin\--oscillators}

   The authors have recently given a general classification of general
   semitoric integrable in dimesion $4$ \cite{san-alvaro-I},
   \cite{san-alvaro-II} in terms of five symplectic invariants; the
   reader familiar with these works can easily that two of these
   invariants do not appear in the case of coupled spin\--oscillators,
   and we state the uniqueness theorem therein in this particular
   case\footnote{The first of these invariants is the number of
     focus\--focus singularities.  The last of these invariants, the
     so called twisting index invariant, is a rather subtle
     topological invariant which measures how the topology near a
     focus\--focus singular fiber relates to the topology near the
     other focus\--focus fibers.  Hence the invariant only appears
     when there is more than one focus\--focus singularity, and in the
     following we shall not mention it. The twisting\--index expresses
     the fact that there is, in a neighborhood of any focus-focus
     point $c_i$, a \emph{privileged toric momentum map} $\nu$. This
     momentum map, in turn, is due to the existence of a unique
     hyperbolic radial vector field in a neighborhood of the
     focus-focus fiber. Therefore, one can view the twisting-index as
     a dynamical invariant. This is an important invariant in the
     general case, see \cite{san-alvaro-I}.}

   Consider a focus\--focus critical point $m$ whose image by
   $(J,\,H)$ is $\tilde{c}$, and let $\Delta$ be a rational convex
   polygon corresponding to the system $(M,\, \omega,\, (J,\,H))$.  If
   $\mu$ is a toric momentum map for the system $(M, \, \omega,\,(J,
   \, H))$ corresponding to $\Delta$, then the image $\mu(m)$ is a
   point in the interior of $\Delta$, along the line $\ell$.  We
   proved in \cite{san-alvaro-I} that the vertical distance
   $h:=\mu(m)-\min_{s \in \ell \cap \Delta} \pi_2(s)>0$ is independent
   of the choice of momentum map $\mu$. Here $\pi_2 \colon \R^2 \to
   \R$ is $\pi_2(c_1,\,c_2)=c_2$.

   \begin{theorem}[consequence of Th.~6.2,
     \cite{san-alvaro-I}] \label{mainthm} Let $(M, \, \omega, \, (J,
     \, H))$ be a $4$\--dimensional semitoric integrable system with
     exactly one focus\--focus singularity.  The {\em list of
       invariants of $(M, \, \omega, \, (J, \, H))$} consists of the
     following items: (i) the Taylor series invariant $(S)^{\infty}$
     at the focus\--focus singularity $m$; (ii) the semitoric polygon
     invariant; (iii) the volume invariant, i.e. the height $h>0$ of
     $m$.  Two $4$\--dimensional simple semitoric integrable systems
     $(M_1, \, \omega_1, (J_1,\,H_1))$ and $(M_2, \, \omega_2,
     (J_2,\,H_2))$ with exactly one focus\--focus singularity are
     isomorphic if and only if the list of invariants (i)\--(iii) of
     $(M_1, \, \omega_1, (J_1,\,H_1))$ is equal to the list of
     invariants (i)\--(iii) of $(M_2, \, \omega_2, (J_2,\,H_2))$.
   \end{theorem}

\begin{theorem}
  The coupled spin\--oscillator has the following symplectic
  invariants: (i) first terms of the Taylor series invariant:
  $a_1=\frac{\pi}{2}$ and $a_2=5\op{ln}2$; (ii) semitoric polygon
  invariant: $(G \times \mathcal{T}) \cdot \Delta_{\scriptop{w}},$
  where $\Delta_{\scriptop{w}}$ is either the upper or lower weighted
  polygon depicted on the right\--most side of Figure~\ref{fig:spin};
  (iii) volume invariant: $h=1$.
\end{theorem}
  
 \begin{proof}
   The semitoric polygon invariant and the first terms of the Taylor
   series invariant were computed previously. The height of the
   focus\--focus point of the system in the polygon is equal to half
   of the Liouville volume of the submanifold of $M$ given by the
   equation $J=1$. This is because the functions $H$ and $J$ are
   symmetric about the $J$\--axis of $\R^2$ in the sense that
   $J(x,\,y,\,z,\,u,\,v)=J(x,\,y,\,z,\,-u,\,-v)$ and
   $H(x,\,y,\,z,\,u,\,v)=-H(x,\,y,\,z,\,-u,\,-v)$. Here there is no
   need to compute anything because the volume of the submanifold
   given by $J=1$ in $M$ is just the length of the vertical slice of
   the polygon at $J=1$, which is $2$, and hence the height of the
   focus\--focus point of the system is $h_1=1$, and the image of the
   focus\--focus point in the polygon is $(1,\,1)$.
 \end{proof}

 \section{Spectral theory for quantum spin\--oscillators}

 In this section, we use the notation of the previous sections
 $J=\frac{u^2+v^2}{2}+z$ and $H=\frac{1}{2}(xu+vy)$. Our goal in this
 section is to quantize this example and analyze its semiclassical
 spectrum.

 First we quickly review the process of assigning a quantum system to
 a classical system. Loosely speaking, a \emph{quantum integrable
   system} is a collection of commuting self\--adjoint operators on a
 Hilbert space.  \emph{Quantization} is a process that takes a
 classical phase space (here, a symplectic manifold $M$) to a Hilbert
 space $\hat{M}$, and classical Hamiltonians $f\in \op{C}^\infty(M)$ to
 self\--adjoint operators $\hat{f}$ acting on $\hat{M}$. The
 quantization of symplectic manifold is often called geometric
 quantization. See the recent book by
 Kostant-Pelayo~\cite{kostant-pelayo} for a survey. Quantizing
 Hamiltonians involves more difficulties. For instance, we need the
 map $f\mapsto \hat{f}$ to be a Lie algebra homomorphism, at least at
 first order~: if the classical system is given by two Poisson
 commuting functions $f,\,g$ then the quantum system is given by two
 operators $\hat{f},\, \hat{g}$ such that
 \begin{equation}
   {\textstyle \frac{\hbar}{i}}[\hat{f},\, \hat{g}]=0 \quad
   \mod(\mathcal{O}(\hbar)).
   \label{equ:commute}
 \end{equation}
 Such a quantization is well-known\footnote{for instance Weyl
   quantization, but there are other possible choices} to exist when
 $M=\R^{2n}$, and more generally on a cotangent bundle $M=\op{T}^*\!X$,
 using $\hbar$-pseudodifferential
 quantization~\cite{dimassi-sjostrand}. Quantizing compact symplectic
 manifolds is also possible under an integrality condition (the
 existence of a so-called prequantum line bundle), using Toeplitz
 quantization~\cite{charles-toeplitz}. However, because of the
 remainder in~\eqref{equ:commute}, it is not known whether a classical
 integrable system can always be quantized to a true quantum
 integrable system. Very recently, in the algebraic setting, the
 relevant obstruction was
 defined~\cite{garay-vanstraten-integrability}. In the coupled
 spin-oscillator example, like in many known systems, an exact
 quantization can be found by hand.

 A well\--known example is the harmonic oscillator in $\R^2$.  The
 harmonic oscillator is given by $M=\R^2$ with coordinates $(u,\,v)$
 and Hamiltonian function on it $ N(u,\,v)=\frac{u^2+v^2}{2}.  $ The
 self\--adjoint operator $\hat{N}$ in the Hilbert space
 $\textup{L}^2(\R)$ given by $
 \hat{N}=-\frac{\hbar^2}{2}\frac{\textup{d}^2}{\textup{d}u^2}+\frac{u^2}{2}
 $ is the standard Weyl quantization of the Hamiltonian $N$. The
 spectrum of $\hat{N}$ is discrete and given by $
 \{\hbar(n+\frac{1}{2}) \, \, | \,\, n \in \N\}.$ The eigenfunctions
 are \emph{Hermite functions}.  This operator will be used as a
 quantum building tool in the sequel.

 \subsection{Quantization of $\R^4$ and the Harmonic Oscillator}
 \label{sec:harmonic_quantization}

 We shall view $S^2$ as a reduced space of $\R^4\simeq\C^2$ under the
 coordinate identification $z_1=x_1+\ii \xi_1$, $z_2=x_2+\ii \xi_2$.
 On $\R^4$ we consider the well\--known harmonic oscillator, $
 L(z_1,\,z_2)=\frac{\vert z_1\vert^2+\vert z_2\vert^2}{2} $ which has
 a $2\pi$-periodic flow generating a Hamiltonian $S^1$\--action $ t
 \cdot (z_1,\,z_2)=(z_1\op{e}^{-\ii t},\, z_2\op{e}^{-\ii t}).$

 The space $Y_E:=\{L=E\}$, for any value $E>0$, is of course the
 euclidean $3$\--sphere $S^3_{\sqrt{2E}}\subset\R^4$ of radius
 $\sqrt{2E}$. It is well known that the reduced space $\{L=E\}/S^1$ is
 $2$\--sphere, and the fibration map $ \{L=E\} \to \{L=E\}/S^1 $ is
 the standard \emph{Hopf fibration}.  More precisely, we may represent
 this reduced space as the euclidean sphere $S^2_{E/2}\subset \R^3$ of
 radius $E/2$.  Denoting by $(x,y,z)$ the variables in $\R^3$, we have
 the following useful formula for the Hopf map, which will be used for
 quantization~:
 \begin{align*}
   \label{eq:4}
   x & = \Re(z_1\bar{z_2})/2 \\
   y & = \Im(z_1\bar{z_2})/2\\
   z & = (\abs{z_1}^2 - \abs{z_2}^2)/4.
 \end{align*}
 
 The usual quantization of $\R^4$ is the Hilbert space $
 \mathcal{H}_{\R^4}=\op{L}^2(\R^2)$.  The Weyl quantization of the
 Hamiltonian function $L$ is the unbounded operator $
 \hat{L}:=-\frac{\hbar^2}{2} \Big(
 \frac{\op{d}^2}{\op{d}\!x_1^2}+\frac{\op{d}^2}{\op{d}\!x_2^2}\Big)
 +\frac{x_1^2+x_2^2}{2}.  $

 The spectrum of $\hat{L}$ is given by $ \op{spec}(\hat{L})=\{\hbar
 (n+1) \, | \, n \in \N\}.$ To see this, define the operator $
 \hat{L}_j:=-\frac{\hbar^2}{2}\Big(\frac{\op{d}^2}{\op{d}\!x_1^2}
 \Big)+\frac{x_j^2}{2} $ acting on $\op{L}^2(\R_{x_j})$. We can write
 $\hat{L}=\hat{L}_1+\hat{L}_2$. Note that the spectrum of $\hat{L}_j$
 is
 \begin{eqnarray} \label{Hjspectrum}
   \op{spec}(L_j)=\{\hbar(n_j+\frac{1}{2})\, | \, n_j \in \N\}.
 \end{eqnarray}
 Therefore we deduce that the spectrum of $\hat{L}$ is given by $
 \{\hbar(n_1+n_2+1) \, | \, n_1 \in \N,\, n_2 \in \N\}, $ and the
 formula above follows since $n_1$ and $n_2$ are arbitrary
 non-negative integers.  The multiplicity of $\hbar(n+1)$ is given by
 the number of pairs $(n_1,\,n_2)$ such that $n_1+n_2=n$, which is
 precisely $n+1$.

 \subsection{Quantization of the space $S^2 \times \R^2$ and the
   Hamiltonians $J$ and $L$}

 We define the \emph{quantization} of $ S^2_{E/2}$ to be the finite
 dimensional Hilbert space $\mathcal{H}_E:=\op{ker}(\hat{L}-E)$.  When
 $E=\hbar(n+1)$, then $\op{dim}(\mathcal{H}_E)=n+1$ (otherwise
 $\mathcal{H}_E=\{0\}$).  It will be convenient to introduce the
 ``anihilation operators'' $a_i:=\frac{1}{\sqrt{2 \hbar}} \Big( \hbar
 \frac{\partial}{\partial x_j} +x_j\Big)$, $i=1,2$, which naturally
 quantize $z_i/\sqrt{2\hbar}$, $i=1,2$ respectively.  Then
 $\hat{L}=\hbar(a_1a_1^*+a_2a_2^*-1)$. The \emph{quantization} of the
 Hamiltonians $x,\,y,\,z$ on $S^2_{E/2}$ are the restrictions to
 $\mathcal{H}_E$ of the operators:
 \begin{eqnarray}
   \hat{x}:=\frac{\hbar}{2}(a_1a_2^*+a_2a_1^*), \qquad
   \hat{y}:=\frac{\hbar}{2\ii }(a_1a_2^*-a_2a_1^*), \qquad
   \hat{z}:=\frac{\hbar}{2}(a_1a_1^*-a_2a_2^*).
 \end{eqnarray}
 This definition makes sense because $\mathcal{H}_E$ is stable under
 the action of $\hat{x},\hat{y},\hat{z}$. This can be checked right
 away using the commutation relations $[a_j,a_j^*]=1$, but it will
 also follow from the explicit action of these operators, as explained
 in Section~\ref{sec:joint} below.

 Of course, in $\R^2_{(u,v)}$, the quantization of $v$ is $\hat{v}:=
 (\frac{\hbar}{\ii }\frac{\partial}{\partial u})$ and the quantization
 $\hat{u}$ of $u$ is the multiplication by $u$ (that we simply denote
 by $u$). Thus we have the very natural definition:

\begin{definition}
  The \emph{quantization} of $S^2_{E/2} \times \R^2$ is the (infinite
  dimensional) Hilbert space $ \mathcal{H}_E \otimes
  \op{L}^2(\R)\subset \op{L}^2(\R^2) \otimes \op{L}^2(\R)$.  The
  \emph{quantization of $J$} is the operator $ \hat{J}=\op{Id} \otimes
  \Big( -\frac{\hbar^2}{2}\frac{\partial^2}{\partial u^2}
  +\frac{u^2}{2} \Big) + (\hat{z} \otimes \op{Id})$.  The
  \emph{quantization of $H$} is the operator $
  \hat{H}=\frac{1}{2}(\hat{x}\otimes u + \hat{y} \otimes
  (\frac{\hbar}{\ii }\frac{\partial}{\partial u})).  $
\end{definition}

This definition depends on the energy $E$, which will be fixed
throughout the paper. For the numerical computations, we have taken
$E=2$, which corresponds to the quantization of the standard sphere
$x^2+y^2+z^2=1$.

\begin{lemma} \label{id0:lem} The operators $\hat{H}$ and $\hat{J}$
  commute, i.e.  we have the identity $[\hat{H},\, \hat{J}]=0$, both
  in the functional analysis sense (\emph{ie.} as an unbounded
  operator on a dense domain), and in the algebraic sense, as a
  bracket in the Lie algebra of polynomial differential operators.
\end{lemma}

\begin{proof}
  It is enough to show that $[\hat{H},\, \hat{J}]=0$ holds on elements
  of the form $f \otimes g$, where $f$ is any element in
  $\mathcal{H}_E$, and $g\in \op{C}^\infty_0(\R)$. And indeed,
  \begin{eqnarray}
    [\hat{H},\, \hat{J}] (f\otimes g) &=& (\hat{H}\hat{J}-\hat{J}\hat{H})(f \otimes g) 
    = \hat{H}\hat{J}(f \otimes g)-\hat{J}\hat{H}(f \otimes g) \nonumber \\
    &=&\hat{H}(f \otimes \hat{N}g +(\hat{z}f)\otimes g)
    -\frac{\hat J}{2}(\hat{x}f\otimes ug +\hat{y}f \otimes \hat{v}g) \nonumber \\
    &=& \frac{1}{2} (\hat{x}f\otimes u \hat{N}g + \hat{x}\hat{z}f
    \otimes ug + \hat{y}f \otimes \hat{v}Ng
    +\hat{y}\hat{\xi}f \otimes \hat{v}g) \nonumber \\
    &-&\frac{1}{2}(\hat{x}f \otimes \hat{N} ug+\hat{y}f \otimes \hat{N}
    \hat{v} g
    +\hat{z}\hat{x}f \otimes ug +\hat{z}\hat{y}f\otimes \hat{v}g) \nonumber \\
    &=&\hat{x}f \otimes [u,\, \hat{N}]g +[\hat{x},\, \hat{z}]f \otimes
    ug + \hat{y}f \otimes [\hat{v},\, \hat{N}]g+[\hat{y},\, \hat{z}]f
    \otimes \hat{v}g. \label{star:for}
  \end{eqnarray}
  As before, we have denoted
  $\hat{N}:=-\frac{\hbar^2}{2}\frac{\partial^2}{\partial
    u^2}+\frac{u^2}2$.  Now
  $$
  [u,\,\hat{N}]f =
  u\Big(-\frac{\hbar^2}{2}\frac{\op{d}^2}{\op{d}\!u^2} +
  \frac{u^2}2\Big)f -
  \Big(-\frac{\hbar^2}{2}\frac{\op{d}^2}{\op{d}\!u^2} +
  \frac{u^2}2\Big)uf = \frac{\hbar^2}{2}\Big( -u
  \frac{\op{d}^2}{\op{d}\!u^2} + \frac{\op{d}^2}{\op{d}u^2}u\Big)f
  $$
  and
 $$
 \frac{\op{d}^2}{\op{d}\!u^2}(uf)=f\frac{\op{d}^2
   u}{\op{d}\!u^2}+2\frac{\op{d} f}{\op{d}\!u}
 \frac{\op{d} u}{\op{d}\!u}+u\frac{\op{d}^2 f}{\op{d}\!u^2}
 =2\frac{\op{d}\!f}{\op{d}\!u}+u\frac{\op{d}^2 f}{\op{d}\!u^2}.
    $$ 
    Hence $ [u,\, \hat{N}]f=\frac{\hbar^2}{2}(2
    \frac{\op{d}\!f}{\op{d}\!u})=\hbar^2\frac{\op{d}}{\op{d}\!u}(f).$
    Therefore $ [u,\, \hat{N}]=\ii {\hbar}\hat{v}.$ Similarly,
    $[\hat{v},\hat{N}]=-\ii \hbar u$. It is also standard to check
    that the ``angular momentum variables'' $(x,y,z)$ satisfy~:
    $[\hat{y},\, \hat{z}]=-\ii {\hbar}\hat{x}$ and $[\hat{x},\,
    \hat{z}]=\ii {\hbar}\hat{y}$.

    Hence expression (\ref{star:for}) equals
$$
\hat{x}f \otimes (\ii \hbar\hat{v})g +(\ii \hbar\hat{y})f\otimes ug
+\hat{y}f \otimes (-\ii {\hbar}u)g + (-\ii \hbar\hat{x})f \otimes
\hat{v}g=0.
$$
The result follows.
\end{proof}

\begin{remark}
  Although the proof of Lemma \ref{id0:lem} is interesting on its own,
  there is a theoretical reason for this lemma to be true, because our
  operators all derive from Weyl quantization of polynomial. And for
  such operators the following result is known: suppose that $H_1$ is
  a quadratic Hamiltonian and $H_2$ is any polynomial Hamiltonian
  function such that $\{H_1,\,H_2\}=0$.  Then Moyal's
  formula~\cite{weyl-qm,moyal,groenewold} yields, formally,
  $[\hat{H}_1,\, \hat{H}_2]=0$.  In our case $J$ is quadratic in the
  variables $(u,v,x_1,x_2,\xi_1,\xi_2)$. This gives an alternative
  proof of Lemma \ref{id0:lem}.
\end{remark}

\subsection{Joint spectrum of $\hat{J}, \, \hat{H}$}
\label{sec:joint}

We have left to find the spectrum of $\hat{H}$ and of
$\hat{J}$. First, we conjugate by the unitary transform in
$\textup{L}^2(\R^2)$~:
\[
U : f(x_1,x_2) \to \sqrt{\hbar}f(\sqrt{\hbar}x_1,\sqrt{\hbar}x_2).
\]
This has the effect of setting $\hbar=1$ in the operator $a_j$~:
\[
U a_j U^* = \frac{1}{\sqrt{2}}\left(\deriv{}{x_j}+x_j\right) =: A_j.
\]

Next, it is convenient to use the Bargmann
representation~\cite{bargmann}, which states that the operator $A_j$
defined above and its adjoint $A_j^*$ are unitarily equivalent to the
operators $\deriv{}{z_j}$ and $z_j$, respectively, acting on the
Hilbert space of holomorphic functions on two variables
$\textup{L}^2_{\textup{hol}}(\C^2,\,
\pi^{-1}\op{e}^{-\abs{z}^2})$. (The notation $z_j$ here is not exactly
the same as the initial one in
section~\ref{sec:harmonic_quantization}, but we keep it for
simplicity.)

The following lemma is standard.
\begin{lemma}[\cite{bargmann}]
  The function $
  \frac{z_1^{\alpha_1}z_2^{\alpha}}{\sqrt{\alpha_1!\alpha_2!}}=\frac{z^{\alpha}}{\sqrt{\alpha!}},
  $ where $\alpha=(\alpha_1,\, \alpha_2)$, is an eigenfunction of
  $\hat{L}$ with norm $1$ and eigenvalue $\hbar(\alpha_1+\alpha_2+1)$.
\end{lemma}
\begin{proof}
  The function $z_i^{\alpha_i}$ is an eigenfunction of
  $z_i\frac{\partial}{\partial z_i}$ with eigenvalue $\alpha_i$. Since
  $\hat{L}=\hbar (z_1 \frac{\partial}{\partial z_1}+z_2
  \frac{\partial}{\partial z_2}+1)$, we get
  $\hat{L}(z^\alpha)=\hbar(\alpha_1+\alpha_2+1)z^\alpha$.

  We can compute
  $\norm{z^{\alpha}}_{\textup{L}^2_{\textup{hol}}(\C^2,\,
    \pi^{-1}\op{e}^{-\abs{z}^2})}^2=\alpha!$.  Therefore the function
  $\frac{z^{\alpha}}{\sqrt{\alpha!}}$ is a normalized eigenfunction of
  $\hat{L}$.
\end{proof}

Next we find the eigenspace of $\hat{L}$ for the eigenvalue
$\hbar(n+1)$. Since the monomials
$\{z^\alpha/\sqrt{\alpha!}\}_{\alpha\in\N^2}$ form a Hilbert basis of
the Bargmann space, the space $\mathcal{H}_E=\ker(\hat{L}-\hbar(n+1))$
is simply given by
$$
\mathcal{H}_E=\op{span}\{\frac{z^{\alpha}}{\sqrt{\alpha!}} \, | \,
\alpha_1+\alpha_2=n\},
$$
thus it is the space of homogeneous polynomials of degree $n$ in
$\mathbb{C}^2$. We will use for it the following basis~:
\[
\{z_2^n,\, z_1^nz_2^{n-1},\, \ldots,\, z_1^{n-1}z_2,\, z_1^n\}.
\]
  
In order to understand the operator $\hat{H}$, we need to consider
$\hat{z}$ and $\hat{N}$.  The restriction of the operator
$\hat{z}=\frac{\hbar}{2}(a_1a_1^*-a_2a_2^*)$ to the Hilbert space
$\mathcal{H}_E$ in given in terms of this polynomial basis by $
\hat{z}(z_1^k z_2^{n-k})=\frac{\hbar}{2} (k-(n-k))z_1^kz_2^{n-k}.  $
It follows that the matrix of
$\hat{z}=\frac{\hbar}{2}(a_1a_1^*-a_2a_2^*)$ relative to this basis is
the diagonal matrix
\[ \frac{\hbar}{2} \left( \begin{array}{ccccccccc}
    -n   & 0    & \dots & &&&0 \\
    0     & 2-n      &  0  &      &               &&  0\\
    0 & 0 &   4-n & 0              && & 0\\
    & &     & &\\
    \vdots & \vdots & \vdots  &\ddots &&\vdots &\vdots\\
    & &                     &  &&&0 \\
    0& 0& \dots & & & 0& n
  \end{array} \right)\]
Notice that this shows that $\mathcal{H}_E$ is indeed invariant
under the action of $\hat{z}$. Of course, a similar calculation can be
done for $\hat{x}$ and $\hat{y}$ (see the proof of
Proposition~\ref{matrix:prop} below). Notice also that the
eigenvalues of 
$\hat{z}$ range from $-\frac{\hbar}{2}n$ to  $\frac{\hbar}{2}n$; 
in the case of the
standard sphere $S^2$ (with $E=2$), we have the relation
$E=2=\hbar(n+1)$. Therefore the eigenvalues of $\hat{z}$ range from
$-\frac{n}{n+1}$ to $\frac{n}{n+1}$. In the semiclassical limit
$n\to\infty$, we recover the classical range $[-1,1]$ of the
hamiltonian $z$ on $S^2$.

Next we consider the Bargmann representation for
$\hat{N}=\frac{\hat{u}^2+\hat{v}^2}{2}$. This time, we act of the
Hilbert space $\textup{L}^2_{\textup{hol}}(\C_{\tau},\,
\pi^{-1}\op{e}^{-\abs{\tau}^2})$ and we obtain $ \hat{N}=\hbar (\tau
\frac{\partial}{\partial \tau}+\frac{1}{2}).  $

The eigenfunctions of $\hat{N}$ are $\frac{\tau^\ell}{\sqrt{\ell!}}$
corresponding to the eigenvalue $\hbar(k+\frac{1}{2})$.

\begin{lemma}
  \label{lemm:J-eigenspace}
  The spectrum of $\hat{J}$ is discrete, and we have
  \[
  \textup{spec}(\hat{J}) = \hbar\left(\frac{1-n}{2} + \N\right).
  \]
  More precisely, for a fixed value $\lambda\in \hbar(\frac{1-n}{2} +
  \N)$, let $\mathcal{E}_{\lambda}:=\ker(\hat{J}-\lambda)$. Then
$$
\mathcal{E}_{\lambda}= \op{span}\Big\{\tau^{\ell} \otimes
z_1^kz_2^{n-k} \,\,\, |\qquad
\hbar(\ell+\frac{1}{2}+k-\frac{n}{2})=\lambda; \quad 0 \le k \leq n;
\quad \ell \geq 0 \Big\}.
$$
In particular $\mathcal{E}_{\lambda}$ has dimension
$1+\min(n,\frac{\lambda}{\hbar} + \frac{n-1}{2})$.
\end{lemma}
\begin{proof}
  In the double Bargmann representation, we have
  \[
  \hat{J} = \op{Id}\otimes (\hbar (\tau \frac{\partial}{\partial
    \tau}+\frac{1}{2})) + \frac{\hbar}{2}(z_1\deriv{}{z_1}-
  z_2\deriv{}{z_2}) \otimes \op{Id}.
  \]
  Hence a simple computation gives
  \begin{eqnarray} \label{eigenfunctions:ex} \hat{J}(\tau^{\ell}
    \otimes z_1^kz_2^{n-k})=
    \hbar\left(\ell+\frac{1}{2}+k-\frac{n}{2}\right) (\tau^{\ell}
    \otimes z_1^kz_2^{n-k})
  \end{eqnarray}
  so the corresponding eigenvalues are
  $\hbar(\ell+\frac{1}{2}+k-\frac{n}{2})$ where $0 \le k \le n$ and
  $n,\,\ell \ge 0$. This shows that $\hat{J}$ admits a complete set of
  eigenvectors. Hence $\ker(\hat{J}-\lambda)$ is spanned by the set of
  eigenvectors coming from this family and corresponding to the
  eigenvalue $\lambda$. This space is finite dimensional (hence
  $\hat{J}$ has discrete spectrum), and its dimension is the number of
  solutions $(k,\ell)$ to the equation
  $\hbar(\ell+\frac{1}{2}+k-\frac{n}{2})=\lambda$ with constraints $0
  \le k \leq n;~~ \ell \geq 0$, which is precisely
  $1+\min(n,\frac{\lambda}{\hbar} + \frac{n-1}{2})$.
\end{proof}

The fact that $\mathcal{E}_{\lambda}$ is finite dimensional should be
compared to the fact that the classical hamiltonian $J$ is proper.

\begin{cor} Given any $n\in\N$, and any $\lambda\in
  \hbar(\frac{1-n}{2} + \N)$, the ordered set
$$
B_{\lambda}:=\Big\{e_{\ell,k}:=\frac{\tau^{\ell}}{\sqrt{\ell!}}
\otimes \frac{z_1^kz_2^{n-k}}{\sqrt{k!(n-k)!}}  \,\,\, |\,\, \, k=0,\,
1,\, \ldots, \op{min}(n,\,
\frac{\lambda}{\hbar}+\frac{n}{2}-\frac{1}{2}),\,\,\,\textup{and}\,\,\,
\ell=\frac{\lambda}{\hbar}+\frac{n}{2}-\frac{1}{2}-k \Big\}.
$$
is an orthonormal basis of $\mathcal{E}_{\lambda}$.
\end{cor}

Our next goal is to compute the matrix of $\hat{H}$. More precisely,
since $\hat{H}$ commutes with $\hat{J}$, the eigenspace
$\mathcal{E}_{\lambda}$ is stable by $\hat{H}$. Thus, the spectral
theory of $\hat{H}$ is merely reduced to the study of the restriction
of $\hat{H}$ to $\mathcal{E}_{\lambda}$, which we explicitly compute
below. Then the best way to depict the spectra of $\hat{J}$ and
$\hat{H}$ is to display the \emph{joint spectrum} (see
figure~\ref{fig:spectrumapprox2}), which is the set of
$(\lambda,\nu)\in\R^2$ such that, for a common eigenfunction $f$, one
has both
\[
\hat{J}f = \lambda f \quad \text{and} \quad \hat{H}f =\nu f.
\]

Let $\ell_0:=\frac{\lambda}{\hbar}+\frac{n}{2}-\frac{1}{2}$,
$\mu=\op{min}(\ell_0,n)$ and let
\begin{eqnarray}
  \beta_k:=\sqrt{(\ell_0+1-k)k(n-k+1)}.\nonumber
\end{eqnarray}

 \begin{figure}[h]
   \centering
   \includegraphics[width=0.5\linewidth]{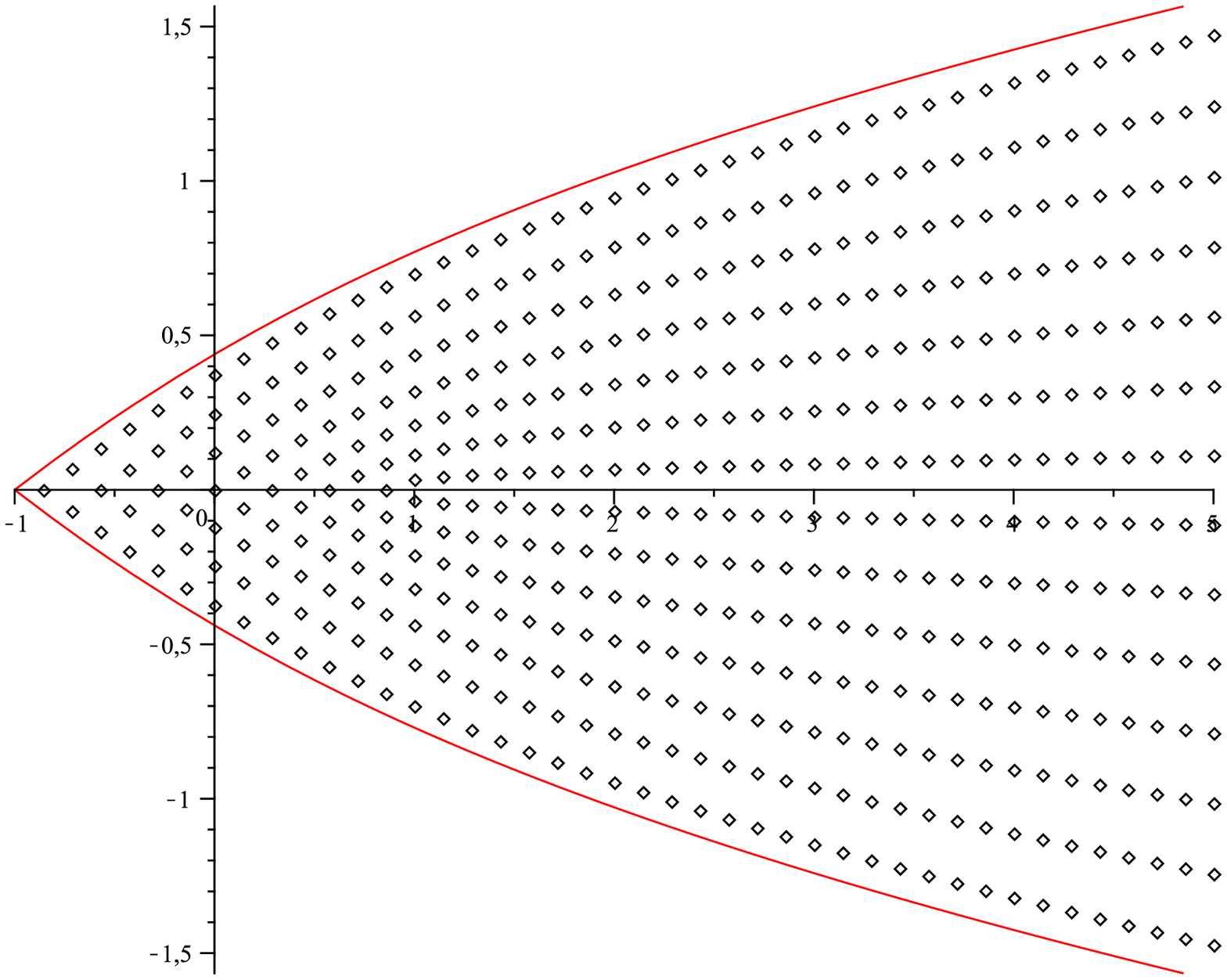}
   \caption{Semiclassical joint spectrum of $\hat{J}, \hat{H}$ and
     momentum map image juxtaposed, computed using a numerical
     diagonalization of the band matrix in
     Proposition~\ref{matrix:prop}. In all our computations we have
     chosen $E=2$, which corresponds to the quantization of the
     standard sphere $x^2+y^2+z^2=1$. This implies the relation
     $2=\hbar(n+1)$. Here $n=13$, so $\hbar\simeq 1.14$.}
   \label{fig:spectrumapprox2}
 \end{figure}

 \begin{prop} \label{matrix:prop} The matrix
   $\textup{M}_{B_{\lambda}}(\hat{H})$ of the self\--adjoint operator
   $\hat{H}$ on the basis $B_{\lambda}$ is the symmetric matrix
   \begin{eqnarray} \nonumber
     \textup{M}_{\mathcal{B}_{\lambda}}(\hat{H})=
     \Big(\frac{\hbar}{2}\Big)^{\frac{3}{2}}
     \left( \begin{array}{ccccccc}
         0   & \beta_1 & \dots & &&&0 \\
         \beta_1     & 0      &  \beta_2 &      &               &&  0\\
         0 &\beta_2&   0 &\beta_3               && & 0\\
         & &     & &\\
         \vdots & \vdots & \ddots  &\vdots &&\vdots &\vdots\\
         & &                     &  &&&\beta_{\mu} \\
         0& 0& \dots & & & \beta_{\mu}& 0
       \end{array} \right).
   \end{eqnarray}
 \end{prop}

\begin{proof}
  We start by evaluating $\hat{x}$ and $\hat{y}$ on this basis:
  \begin{eqnarray}
    \hat{x}(z_1^kz_2^{n-k})&=&\frac{\hbar}{2}(kz_1^{k-1}z_2^{n-k+1}+(n-k)z_1^{k+1}z_2^{n-k-1}) 
    \nonumber \\
    \hat{y}(z_1^kz_2^{n-k})&=&\frac{\hbar}{2\ii }(kz_1^{k-1}z_2^{n-k+1})-
    (n-k)z_1^{k+1}z_2^{n-k-1}) \nonumber
  \end{eqnarray}

  We introduce:
$$
\alpha:=\frac{1}{\sqrt{2\hbar}}(u+\hbar \frac{\partial h}{\partial
  u}), \qquad \alpha^*:=\frac{1}{\sqrt{2\hbar}}(u-\hbar \frac{\partial
  h}{\partial u})
$$

Hence $u(=\hat{u})=(\alpha+\alpha^*)\sqrt{\frac{\hbar}{2}}$.  Now we
do the Bargmann representation
\begin{eqnarray}
  \hat{u}=\sqrt{\frac{\hbar}{2}}(\tau+\frac{\partial}{\partial
    \tau}), \qquad 
  \hat{v}=\frac{\hbar}{\ii }\frac{\partial}{\partial u}=
  \frac{(\alpha-\alpha^*)}{\ii}\sqrt{\frac{\hbar}{2}}=\frac{1}{\ii} \nonumber
  \sqrt{\frac{\hbar}{2}}(\frac{\partial}{\partial \tau}-\tau).
\end{eqnarray}

Hence we obtain
\begin{eqnarray}
  \hat{u}(\tau^{\ell})=\sqrt{\frac{\hbar}{2}} (\tau^{\ell+1}+\ell \tau^{\ell-1}),\,\,\,\,\,\,\,\,\,\,\,
  \hat{v}(\tau^{\ell})&=&\frac{1}{\ii }\sqrt{\frac{\hbar}{2}} (\ell \tau^{\ell-1}-\tau^{\ell+1}). \nonumber
\end{eqnarray}

In what follows, for brevity of the notation, we write
$c_k:=z_1^kz_2^{n-k}$. Note that $n$ is fixed.  Recalling
$\hat{H}=\frac{1}{2}(\hat{u}\otimes \hat{x}+\hat{v}\otimes\hat{y})$,
we get
\begin{eqnarray}
  \hat{H}(\tau^{\ell}z_1^kz_2^{n-k})&=&
  \frac{1}{2} \Big( \Big(\frac{\hbar}{2}\Big)^{3/2}
  (\tau^{\ell+1}+\ell \tau^{\ell-1})(kc_{k-1}+(n-k)c_{k+1}) \nonumber \\
  &-&\Big( \frac{\hbar}{2} \Big)^{3/2} (\ell \tau^{\ell-1}-\tau^{\ell+1})
  (kc_{k-1}-(n-k)c_{k+1})  \Big)
  \nonumber \\
  &=& \frac{1}{2}\Big(\frac{\hbar}{2}\Big)^{3/2}\Big(
  k\tau^{\ell+1}c_{k-1}+\ell k \tau^{\ell-1}c_{k-1}
  +(n-k)\tau^{\ell+1}c_{k+1}
  +\ell(n-k)\tau^{\ell-1}c_{k+1} \nonumber \\
  &-&\ell k\tau^{\ell-1}c_{k-1}+\ell(n-k)\tau^{\ell-1}c_{k+1}
  +k \tau^{\ell+1}c_{k-1}-
  (n-k)\tau^{\ell+1}c_{k+1}\Big) \nonumber \\
  &=& \Big(\frac{\hbar}{2}\Big)^{3/2}
  (k\tau^{\ell+1}c_{k-1}+(n-k)\ell \tau^{\ell-1} c_{k+1}). \label{kj}
\end{eqnarray}

Notice how this formula, together with Lemma~\ref{lemm:J-eigenspace},
confirms that $\mathcal{E}_\lambda$ is stable under $\hat{H}$.

In order to have a better numerically prepared matrix (and a
nicer-looking formula !), we next express everything in an orthonormal
basis.  Denote $ e_{\ell,k}=\frac{\tau^{\ell}}{\sqrt{\ell!}}
\frac{z_1^kz_2^{n-k}}{\sqrt{k! (n-k)!}}  $ so that $e_{\ell,k}$ is an
eigenvector of $\hat{J}$ of norm $1$:
\begin{eqnarray}
  \hat{J}(e_{\ell,k})&=&\hbar (\ell+\frac{1}{2}+k-\frac{n}{2})
  e_{\ell,k} = \lambda  e_{\ell,k}\nonumber \\
  \hat{H}(e_{\ell,k})&=& \Big(\frac{\hbar}{2}\Big)^{3/2} 
  \frac{k \tau^{\ell+1}c_{k-1}+\ell(n-k)\tau^{\ell-1}c_{k+1}}{\sqrt{\ell!k! (n-k)!}}. \label{tr}
\end{eqnarray}

On the other hand we have that $
e_{\ell+1,k-1}=\frac{\tau^{\ell+1}c_{k-1}}{\sqrt{(\ell+1)!(k-1)!(n-k+1)!}}
$ and that the first term of (\ref{tr}) is
\begin{eqnarray}
  \frac{k}{\sqrt{\ell! k! (n-k)!}}\tau^{\ell+1}c_{k-1}&=&\frac{k}{\sqrt{\ell!k!(n-k)!}}
  \sqrt{(\ell+1)!(k-1)!(n-k+1)!} e_{\ell+1,k-1} \nonumber \\
  &=&\sqrt{(\ell+1)k(n-k+1)} e_{\ell+1,k-1}. \nonumber 
\end{eqnarray}

Similarly the second term of (\ref{tr}) is
\begin{eqnarray}
  \frac{\ell (n-k)\tau^{\ell-1}c_{k+1}}{\sqrt{\ell! k! (n-k)!}}
  &=&\frac{\ell(n-k)}{\sqrt{\ell!k!(n-k)!}} \sqrt{(\ell-1)!(k+1)!(n-k-1)!} e_{\ell-1,k+1} \nonumber \\ 
  &=&\sqrt{\ell(k+1)(n-k)}e_{\ell-1,k+1}. \nonumber 
\end{eqnarray}

Since $\ell=\ell_0-k$, we get
\begin{eqnarray*}
  \hat{H}(e_{\ell,k}) & = &\Big( \frac{\hbar}{2}\Big)^{3/2}
  \Big( \sqrt{(\ell_0-k+1)k(n-k-1)} e_{\ell+1,k-1}+
  \sqrt{(\ell_0-k)(k+1)(n-k)}e_{\ell-1,k+1}
  \Big)\\
  & = & \Big( \frac{\hbar}{2}\Big)^{3/2} (\beta_k e_{\ell+1,k-1} +
  \beta_{k+1} e_{\ell-1,k+1}).
\end{eqnarray*}

This, of course, gives the statement of the proposition.
\end{proof}

\subsection{The spectrum $\Sigma(n)$ of
  $\hat{H}|_{\op{ker}(\hat{J}-\op{Id})}$}

In the next section, we will be particularly interested in the
$\hat{J}$-eigenvalue $\lambda=1$, which corresponds to the
$J$-critical value of the focus-focus point, in the classical
system. Since $E=2=\hbar(n+1)$, we see that $\ell_0=\frac{n+1}{2} +
\frac{n-1}{2}=n$. Therefore the dimension of $\ker(\hat{J}-\op{Id})$
is equal to $n+1$. Notice that, for $\lambda<1$, the dimension of
$\ker(\hat{J}-\lambda)$ is increasing linearly with slope 1 (with
respect to the parameter $k$ that we introduced above) whereas for
$\lambda>1$ this dimension is constant, equal to $n+1$. This can be
seen as a quantum manifestation of the Duistermaat-Heckmann
formula~\cite{duist-heckman}.

\section{Inverse spectral theory for quantum spin\--oscillators}

The theme of this section is to give evidence of the following
conjecture being true in the case of coupled spin oscillators:

 \begin{conj}
   A semitoric system is determined up to symplectic equivalence by
   its semiclassical joint spectrum (\emph{i.e.}  the set of points in
   $\mathbb{R}^2$ where on the $x$\--axis we have the eigenvalues
   $\lambda$ of $\hat{J}$, and on the vertical axes the eigenvalues of
   $\hat{H}$ restricted to the $\lambda$\--eigenspace of $\hat{J}$).
   From any such spectrum one can construct explicitly the associated
   semitoric system.
 \end{conj}

 In this section we try to convey some ideas to explicitly compute all
 the symplectic invariants from the semiclassical spectrum. It might
 not necessarily be the optimal way to prove an inverse spectral
 result, as some quantities are more easily defined implicitly rather
 than explicitly by the spectrum. But we believe that, from a quantum
 viewpoint, having constructive formulas for the symplectic invariants
 is particularly valuable.

 We emphasize the word ``semiclassical'' here~: in order to recover
 the symplectic invariants we need be able to compute the joint
 spectrum for small values of $\hbar$. What can be said for a unique,
 fixed value of $\hbar$ is much harder question.

 \subsection{Polygon and height invariant}

 Recovering the polygon invariant is probably the easiest and most
 pictorial procedure, as long as one stays on a heuristic
 level. Making the heuristic rigorous should be possible along the
 lines of the toric case explained in~\cite{san-panoramas} and
 ~\cite{san-inverse}, but we don't attempt to do it here.

 The first thing to do is to recover the image of the classical moment
 map, including the position of the singular values. This could be
 done by a local examination of density of the joint eigenvalues.

 Next, in order to recover the polygon invariant, we need to obtain
 the integral affine structure of the image of the momentum map. We
 know from~\cite{duist-cushman, san-panoramas} that the joint spectrum
 possesses a semiclassical integral affine structure on the regular
 values of the momentum map. This integral affine structure can be
 extended to the elliptic boundaries, as explained
 in~\cite{san-panoramas}. Thus, except along a vertical cut through
 the focus-focus critical value, one can develop this affine structure
 such that the joint eigenvalues become elements of the lattice $\hbar
 \Z^2$. See figure~\ref{fig:recover_polytope}.
 \begin{figure}[H]
   \centering
   \includegraphics[width=0.5\textwidth]{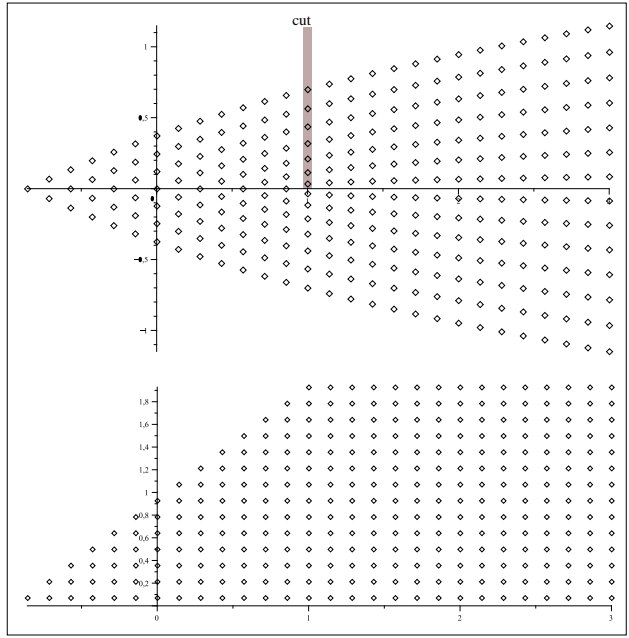}
   \caption{Recovering the polygon invariant. The top picture is the
     joint spectrum of $(\hat{J},\hat{H})$. In the bottom picture, we
     have developed the joint eigenvalues into a regular lattice. One
     can easily check on this illustration that the number of
     eigenvalues in each vertical line in the same in both pictures.}
   \label{fig:recover_polytope}
 \end{figure}
 The convex hull of the resulting set is a rational, convex polygonal
 set, depending on $\hbar$. Since the semiclassical affine structure
 is an $\hbar$-deformation of the classical affine structure, we see
 that, as $\hbar\to 0$, this polygonal set converges to the semitoric
 polygon invariant.

 \subsection{Semiclassical formula for the spectrum $\Sigma(n)$}

 In order to recover the Taylor series invariant from the spectrum, we
 need a precise description of this spectrum. There are two options~:
 either describe the spectrum in regular regions, and then take the
 limit to the focus-focus critical value; or describe the spectrum
 directly in a small neighborhood of the focus-focus value. We choose
 the second option, because it seems more appropriate for a reasonably
 accurate numerical formula for the invariants, in the spirit of
 equation~\eqref{equ:numerical_formula}.

 The drawback of this approach is that there is no result currently
 available giving the description of this spectrum. The singular
 Bohr-Sommerfeld rules of~\cite{san-focus} would give the required
 result, in case $\hat{J}$ and $\hat{H}$ were pseudodifferential
 operators. Of course they are not, since the phase space $S^2\times
 \R^2$ is not a cotangent bundle. However they are semiclassical
 Toeplitz operators, in the sense of~\cite{charles-toeplitz}, and it
 is known that the algebra of Toeplitz operators is microlocally
 equivalent to the algebra of pseudodifferential
 operators~\cite{BG}. Therefore, we propose the following conjecture.
 \begin{conj}
   The formula in Corollary 6.8 in V\~ u Ng\d oc's paper
   \cite{san-focus} holds also if the operators therein involved are
   Toeplitz instead of pseudodifferential.
 \end{conj}

 This conjecture may be stated in the following way. Let $\Sigma(n)$
 be the spectrum of $\hat{H}|_{\op{ker}(\hat{J}-\op{Id})}$.  For
 bounded $t\in\R$, the formula
$$
\tilde{\lambda}(t) -\tilde{\epsilon}(t)
\op{ln}(2\hbar)-2\op{arg}\Gamma \Big( \frac{\ii
  \tilde{\epsilon}(t)+1+j}{2} \Big) \, \in 2\pi \mathbb{Z}
+\mathcal{O}(\hbar^{\infty})
$$
holds if and only if $\hbar t \in
\Sigma(n)+\mathcal{O}(\hbar^{\infty})$ with
\begin{itemize}
\item[(a)] $\tilde{\lambda}(t)=\tilde{\lambda}(t;\hbar)$ admits an
  asymptotic expansion on integer $\ge -1$ powers of $\hbar$ with
  smooth (=$\op{C}^{\infty}$) coefficients in $t$ starting with
  $\tilde{\lambda}(t)=\frac{1}{\hbar} \int_{\gamma_0} \alpha_0
  +\op{I}_{\gamma_0}(\tilde{\kappa}(t))+\mu
  \frac{\pi}{2}+\mathcal{O}(\hbar).  $

\item[(b)] $\tilde{\epsilon}(t)=\tilde{\epsilon}(t;\hbar)$ has an
  asymptotic expansion on integer $\ge 0$ powers of $\hbar$ with
  smooth coefficients in $t$ starting with the second component of the
  vector $B (0,\, t) + \mathcal{O}(\hbar) $ where $B$ is the $2\times
  2$ matrix such that $B (J'',\, H'')_m=(q_1,\, q_2)$.

\item[(c)] $I_{\gamma_0}(\tilde{\kappa}(t))$ is what is called the
  ``principal value integral'' of $\tilde{\kappa}(t)$, where
  $\tilde{\kappa}(t)$ is the $1$\--form on $\Lambda_0$ defined by
  \begin{eqnarray} \label{for:tildekappa}
    (\tilde{\kappa}(t)(\mathcal{X}_J),\,
    \tilde{\kappa}(t)(\mathcal{X}_H))=(0,\,t) \iff 
    (\tilde{\kappa}(t)(\mathcal{X}_{q_1}), \,
    \tilde{\kappa}(t)\mathcal{X}_{q_2}))=B (0,\, t)
  \end{eqnarray}
  Finally, $I_{\gamma_0}(\tilde{\kappa}^t)$ is defined in Proposition
  6.15 of \cite{san-focus} as
$$
I_{\gamma_0}(\tilde{\kappa}(t))=\lim_{(s_1,\,s_2) \to (0,\,0)} \Big(
\int_{A_0=\gamma_0(s_1)}^{B_0=\gamma_0(1-s_2)} \tilde{\kappa}(t)
+\epsilon(t) \op{ln}(r_{A_0}\rho_{B_0}) \Big)
$$
where $\epsilon(t)$ is the first order term of $\tilde{\epsilon}(t)$.
\end{itemize}

For a semitoric system, the matrix $B$ is of the form
$B=\begin{pmatrix}
  1 & 0\\
  B_{21} & B_{22}
\end{pmatrix}$, with $B_{22}\neq 0$. Thus we get
\[
\epsilon(t)=B_{22}t.
\]

Moreover, because of formula (\ref{for:tildekappa}),
\[
(\tilde{\kappa}(t)(\mathcal{X}_{q_1}), \,
\tilde{\kappa}(t)(\mathcal{X}_{q_2}))=(0,\, B_{22}t).
\]
Therefore we see that
$\deriv{\tilde{\kappa}(t)}{t}=B_{22}\kappa_{2,0}$, where
$\kappa_{2,0}$ is the restriction to $\Lambda_0$ of the 1-form defined
in equation~\eqref{for:kappa2}. Thus, in view of
equation~\eqref{equ:formula2}, we get an explicit formula for the
symplectic invariant $a_2$~:
\begin{eqnarray} \label{for:a1} a_2 =
  \frac{1}{B_{22}}\frac{\partial}{\partial t}\Big(
  I_{\gamma_0}(\tilde{\kappa}^t) \Big)\upharpoonright_{t=0}.
\end{eqnarray}

Though we haven't worked it out here, a similar formula for the first
invariant $a_1$ could be obtained along the same lines.
 
In the case of the coupled spin\--oscillator, $B=\begin{pmatrix}
  1 & 0\\
  0 & 2
\end{pmatrix}$, so $B_{22}=2$ and $a_2 =
\frac{1}{2}\frac{\partial}{\partial t}(
I_{\gamma_0}(\tilde{\kappa}^t))\upharpoonright_{t=0}$.

\subsection{Obtaining $a_2$ from the spectrum $\Sigma(n)$}

We show in this paragraph how the conjecture gives a way to obtain
$a_2$. Using formula~\eqref{for:a1} above, an easy corollary of the
conjecture is Theorem 7.6 in \cite{san-focus}, which says that
\begin{equation}
  \min\Big(\frac{E_{k+1}-E_k}{\hbar} \Big) =\frac{2\pi/B_{22}}{\vert
    \op{ln}\hbar \vert+a_2 +\op{ln}2+\gamma}+\mathcal{O}(\hbar)
  \label{eq:spacing}
\end{equation}
for $\Sigma(n)=\{E_0 \le E_1 \le \ldots \le E_n\}$.  Here $\gamma$ is
Euler's constant. 

From the spectrum we can calculate $
t^{\min}(\hbar)=\min\Big(\frac{E_{k+1}-E_k}{\hbar} \Big) $ so
$$
\frac{2\pi}{t^{\min}}=B_{22}(\vert \op{ln}\hbar \vert
+a_2+\op{ln}2+\gamma)(1+\mathcal{O}(\hbar)) =B_{22}(\vert \op{ln}\hbar
\vert+a_2+\op{ln}2+\gamma)+\mathcal{O}(\hbar \op{ln}\hbar).
$$
Therefore we may recover $B_{22}$ as
\begin{equation}
  B_{22} = \lim_{\hbar\to 0}
  \left(\frac{2\pi}{t^{\min}\abs{\ln\hbar}}\right).
  \label{eq:B22}
\end{equation}
Because the convergence of this limit is very slow (of order
${\abs{\ln\hbar}}^{-1}$), it is in practice much better to solve the
system obtained with two different values of $\hbar$, which gives~:
\begin{equation}
  \label{eq:B22-bis}
  B_{22} = \frac{\frac{2\pi}{t^{\min}(\hbar_1)} -
      \frac{2\pi}{t^{\min}(\hbar_2)}}{\ln(\hbar_2/\hbar_1)} +
    \mathcal{O}(\hbar_1\ln \hbar_1) + \mathcal{O}(\hbar_2\ln \hbar_2).
\end{equation}
Thus, if we choose $\hbar_2$ to be a fixed multiple of
$\hbar=\hbar_1$, we get a convergence speed of order
$\mathcal{O}(\hbar\ln \hbar)$, which is indeed much more reasonable.

Once $B_{22}$ is known, it is easy to recover $a_2$, again through
formula~\eqref{eq:spacing}~:
\begin{equation}
  \label{eq:compute-a1}
  a_2 = \lim_{\hbar\to 0}
  \left(\frac{2\pi}{B_{22}t^{\min}}-\abs{\ln\hbar} - \ln 2 - \gamma\right), 
\end{equation}
and the convergence rate is again of order $\mathcal{O}(\hbar\ln
\hbar)$.

\subsection{Numerical approximation of $a_2$ using Maple}

Using Proposition~\ref{matrix:prop}, we compute the spectrum
$\Sigma(n)$ of the Spin-Oscillator example for various values of
$n=2/\hbar -1$ by entering the matrix in the computer algebra system
'Maple' and ask for a numeric diagonalization. Then is it easy to
implement the formulas~\eqref{eq:B22-bis} and~\eqref{eq:compute-a1}.

From the general theory, the minimal eigenvalue spacing is obtained
--- at least in the limit $\hbar\to 0$, at the focus-focus critical
value $H=0$. This is confirmed from the numerics. In fact, using the
recursion formula for the characteristic polynomial $D_n(X)$ of the
matrix $\textup{M}_{B_{\lambda}}(\hat{H})$ (with $\ell_0=n$)~:
\[
D_n(X) = XD_{n-1}(X)-\beta_n^2 D_{n-2}(X),
\]
we prove by induction that $D_n(X)$ has the parity of $n+1$. In
particular, the spectrum is symmetric~: $\Sigma(n)=-\Sigma(n)$. When
$n$ is odd, $0$ is not an eigenvalue
($D_n(0)=(-1)^{(n-1)/2}\beta_1\beta_3\cdots \beta_n$), and hence the
smallest spacing is simply twice the smallest positive eigenvalue~:
\[
t^{\min}(\hbar) = 2E_{[\frac{n}{2}]+2}/\hbar \quad \text{ with }
\hbar=\frac{2}{n+1}.
\]
\begin{figure}[H]
  \centering
  \includegraphics[width=0.7\linewidth]{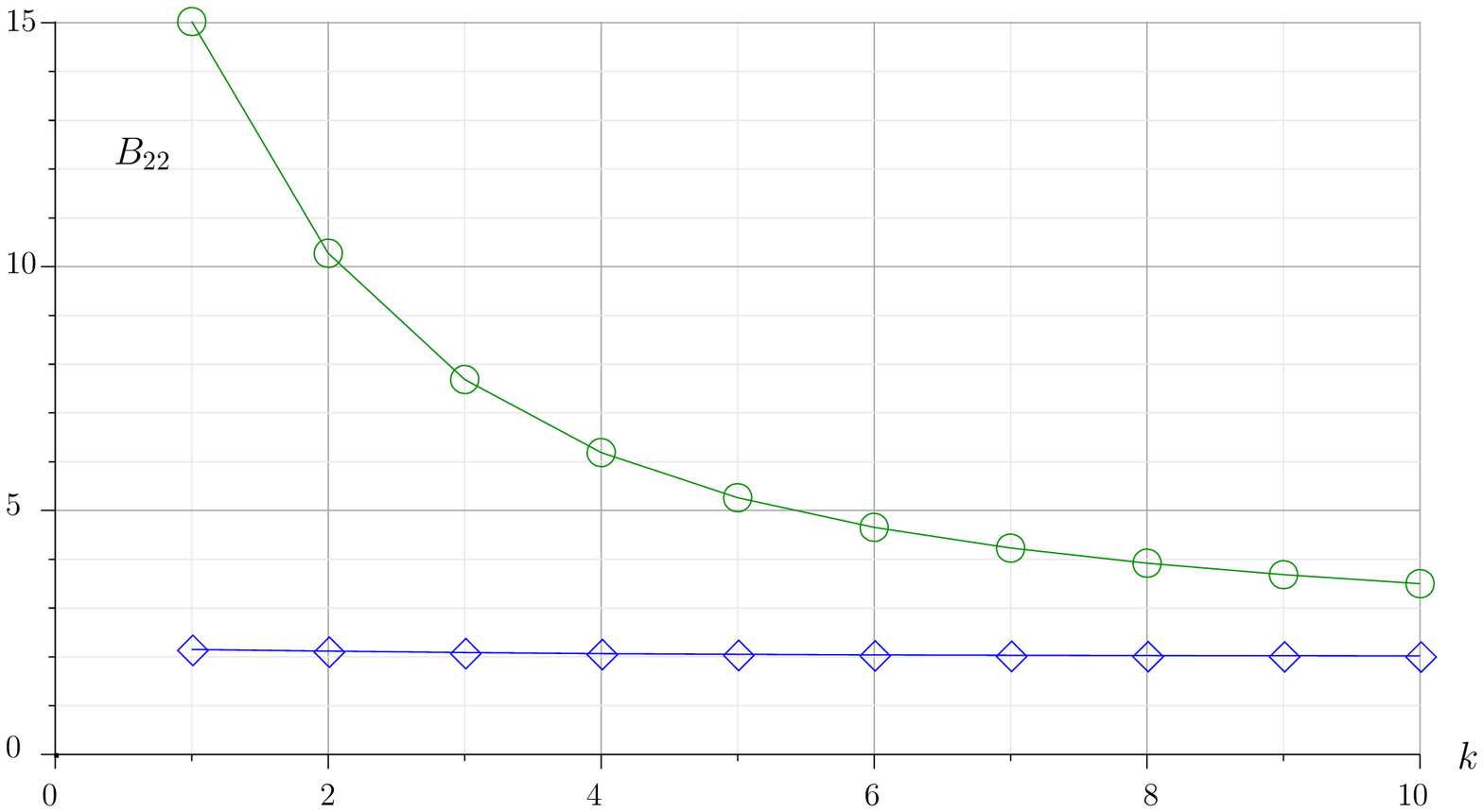}
  \caption{Recovering the coefficient $B_{22}$ (which is equal to 2 in
    our example). The horizontal scale is logarithmic: the integer
    abscissa $k$ corresponds to $n=2^k+1$. Thus $\hbar$ starts at $0.5$
    and decreases to the right to reach $1/513\simeq 0.002$.  The top
    curve --- with circles --- is the result of
    formula~\eqref{eq:B22}, which indeed converges very slowly. The
    curve with diamonds is obtained by the accelerated
    formula~\eqref{eq:B22-bis}.}
  \label{fig:b22}
\end{figure}
\begin{figure}[H]
  \centering
  \includegraphics[width=0.7\linewidth]{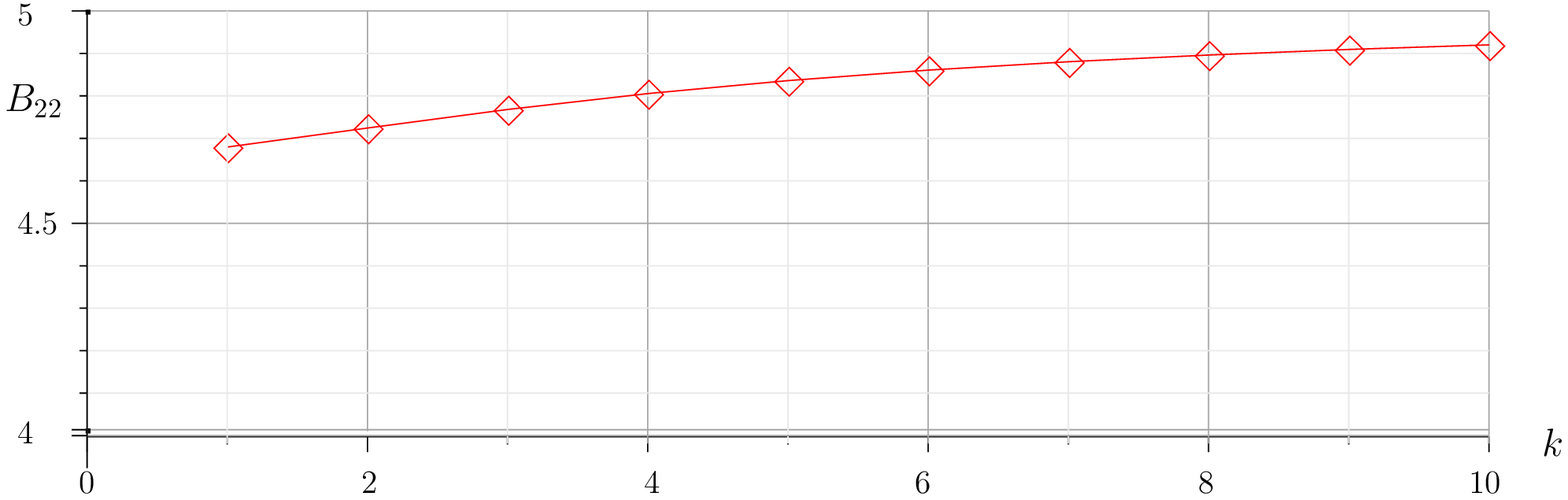}
  \caption{Recovering the invariant $a_2$. The graph plots the values
    of $a_2/\ln 2$ (which should be 5 in our example) computed using
    the formula~\eqref{eq:compute-a1}. The horizontal scale is the
    same is in figure~\ref{fig:b22}.}
  \label{fig:a1}
\end{figure}
The results of our numerical experiments are plotted in
figures~\ref{fig:b22} and~\ref{fig:a1}. They should be compared to the
theoretical values of Theorem~\ref{theo:a1}.

\noindent
\\
Alvaro Pelayo\\
University of California\---Berkeley \\
Mathematics Department \\
970 Evans Hall $\#$ 3840 \\
Berkeley, CA 94720-3840, USA.\\
{\em E\--mail}: {apelayo@math.berkeley.edu}

\bigskip\noindent

\noindent
V\~u Ng\d oc San\\
Institut de Recherches Math\'ematiques de Rennes\\
Universit\'e de Rennes 1\\
Campus de Beaulieu\\
35042 Rennes cedex (France)\\
{\em E-mail:} {san.vu-ngoc@univ-rennes1.fr}

\end{document}